\documentclass[11pt]{amsart}
\usepackage{amssymb,epsfig,amsmath,latexsym,amsthm}
\usepackage{graphicx}
\usepackage{epsfig,color}
\usepackage{enumitem}
\usepackage{comment}
\usepackage{tikz-cd}
\theoremstyle{plain}
\newtheorem{theorem}{Theorem}[section]
\newtheorem{lemma}[theorem]{Lemma}

\newtheorem{proposition}[theorem]{Proposition}
\newtheorem{corollary}[theorem]{Corollary}
\newtheorem{question}[theorem]{Question}

\theoremstyle{definition}
\newtheorem{definition}[theorem]{Definition}
\newtheorem{remark}[theorem]{Remark}

\newtheorem{example}[theorem]{Example}

\usepackage{epsfig,color}

\makeatother

\theoremstyle{definition}

\def\fnum{equation}

\numberwithin{equation}{section}

\usepackage[hidelinks]{hyperref}

\usepackage{tikz-cd}

\usepackage[paperheight=11in,paperwidth=8.5in,left=1in,top=1.25in,right=1in,bottom=1.25in,]{geometry}
\begin{document}
\title
[Isoperimetric interpretation for Renormalized volume]
{Isoperimetric interpretation for the Renormalized volume of convex co-compact hyperbolic $3$-manifolds}

\address{
	Department of Mathematics, Yale University, New Haven, CT 08540, USA
}
\email{franco.vargaspallete@yale.edu}

\address{
				Universidade Federal de Minas Gerais, 			Belo Horizonte 30123-970, Brazil
} 
\email{celso@mat.ufmg.br}

\author{Franco Vargas Pallete and Celso Viana}
\thanks{FVP was supported by the Minerva Research Foundation and by NSF grant DMS-2001997. Part of this material is also based upon work supported by the National Science Foundation under Grant No. DMS-1928930 while FVP participated in a program hosted by the Mathematical Sciences Research Institute in Berkeley, California, during the Fall 2020 semester.}

\begin{abstract}
We reinterpret renormalized volume as the asymptotic difference
of the isoperimetric profiles for convex co-compact hyperbolic 3-manifolds. By similar techniques
we also prove a sharp Minkowski inequality for horospherically convex sets in $\mathbb{H}^3$. Finally, we include
the classification of stable constant mean curvature surfaces  in regions bounded by two  geodesic planes in $\mathbb{H}^3$ or in cyclic quotients of $\mathbb{H}^3$.
\end{abstract}

\maketitle

\section{Introduction}

\noindent Renormalized volume is a geometric quantity motivated by the AdS/CFT correspondence and the calculation of gravitational action (see Witten \cite{W}). For a convex co-compact hyperbolic 3-manifold, one can use the duality done by Epstein \cite{E} between conformal metrics at infinity ($\partial_{\infty} \mathbb{H}^3$) and immersions into $\mathbb{H}^3$ to construct a submanifold $N\subset M$ so that $V_R(M)$ is equal to the volume of $N$ minus half of the integral of the mean curvature of $\partial N$. In this setup,  renormalized volume can be characterized as the antiderivative of a 1-form defined by the Schwarzian derivative of the uniformization maps at infinity.  See Section 3 for precise definition of renormalized volume addressed here and  further discussion. This present work reinterprets renormalized volume of an acylindrical hyperbolic manifold $M$ as the asymptotic difference between the isoperimetric profile of $M$ and the isoperimetric profile of the representative with totally geodesic convex core in the deformation space of $M$. An \textit{isoperimetric profile} of a given manifold is a function that for each number $V>0$ assigns the optimal perimeter of a region of volume $V$. There are two profiles we consider: $I_M(V)$ and $J_M(V)$, the outermost isoperimetric profile and the isoperimetric profile, respectively. While $J_M(V)$ is taken without restriction, $I_M(V)$ is defined taking optimal perimeters between regions containing all compact minimal surfaces in $M$. This requires the competitor for $I_M(V)$ to contain a region with minimal boundary, which we call the \textit{outermost region} of M.
 
\begin{theorem}
	Let $M$ be a convex co-compact hyperbolic $3$-manifold that is either acylindrical or quasifuchsian and let $\Omega_0$ be its outermost region. If $I_{M}, I_{TG}$ denote the outermost isoperimetric profiles of $M$ and $M_{TG}$ (the quasiconformal deformation of $M$ with Fuchsian ends) respectively, then
	\begin{eqnarray*}
		V_R(M)-|\Omega_0|=	\frac12\lim_{V\rightarrow \infty} \big(I_{TG}(V)- I_{M}(V)\big).
	\end{eqnarray*}
\end{theorem}
\noindent See Section \ref{sec:isomperimetricconvexcocompact} for precise definitions and properties of $I_M, J_M, M_{TG}$ and $\Omega_0$.

\noindent It is  well known that convex co-compact hyperbolic $3$-manifolds have near infinity a foliation with CMC leaves \cite{MP}. In Theorem \ref{foliation is isoperimetric} we show that such foliation is in fact isoperimetric. Using this fact we prove  that the renormalized volume  is determined by the geometric data of that CMC foliation:
\begin{theorem} Let $M$ be a convex co-compact hyperbolic 3-manifold. If $J_M$ is the isoperimetric profile of $M$, then
	\begin{eqnarray*}
		V_R(M) +\frac{\pi}{2}\chi(\partial M) = \lim_{V\rightarrow \infty} \bigg( 
		V-  \frac{1}{2}J_{M}(V) + \pi \chi(\partial M)\log \sqrt{\frac{2\,J_M(V)}{\pi |\chi(\partial M)|}}\,
		\bigg).
	\end{eqnarray*}
\end{theorem}

\noindent The volume-comparison interpretation of renormalized volume is on a similar spirit to the notions defined for asymptotically hyperbolic manifolds, see for instance the one studied by Brendle
and Chodosh \cite{BC} (see also \cite{JSZ}).  The work \cite{JSZ} also  proves a sharp  isoperimetric  comparison result for AH $3$-manifolds with scalar curvature $R\geq -6$.   We prove the following  comparison for the isoperimetric profile $I_{TG}$ of the convex-co compact hyperbolic $3$-manifold with totally geodesic convex core:

\begin{theorem}\label{global comparison}
	Let $M$ be a convex co-compact hyperbolic $3$-manifold that is either acylindrical or quasifuchsian and $\Omega_0$ its outermost region. 	If	$V_R(M)>|\Omega_0|$, then $I_M(V)<I_{TG}(V)$ for every volume $V\geq0$. 
	
If $M$ contains only one minimal surface, then $I_M(V)<I_{TG}(V)$ for every volume $V\geq0$.
\end{theorem}

\noindent The authors in \cite{JSZ} use inverse mean curvature flow to produce a candidate profile that obeys the comparison inequality; the positivity of the Hawking mass plays an important role in their proof.  Theorem \ref{global comparison} is  based on the analytic features of the isoperimetric profile. One fundamental difference is the change of sign of the Euler characteristic of the boundary, and subsequently of the Hawking mass, since this changes the convexity/concavity properties we have at our disposal. On one side, we notice that the Hawking mass, being negative in our setting, allows the difference of profiles to have a positive local  maximum. On the other hand, it forbids a negative local  minimum. Hence one of our challenges is to successfully use these inverted properties for the Hawking mass in order to conclude a profile camparison. In some perspective Theorem \ref{global comparison}  reflects a result proved in \cite{BBB,VP} concerning the infimum  of the renormalized volume as a functional in the moduli space of convex co-compact hyperbolic $3$-manifolds.
\vspace{0.1cm}

\noindent We apply the duality  \cite{E}, relating horospherically convex sets in $\mathbb{H}^3$ with conformal metrics at infinity, and the renormalized Ricci flow to also prove  a sharp Minkowski type inequality that characterizes geodesic balls in $\mathbb{H}^3$. 
\begin{theorem}
	If $\Sigma$ is an horospherically convex surface bounding a compact region $\Omega \subset \mathbb{H}^3$, then
	\begin{eqnarray*}
		\int_{\Sigma} H\,d\Sigma - 2|\Omega| \, \geq\, 2\pi \log\bigg(1 + \frac{1}{2\pi}\int_{\Sigma}(H+1)d\Sigma \bigg) 
	\end{eqnarray*}
	with equality if, and only if, $\Sigma$ is a geodesic sphere.
\end{theorem}
\noindent This inequality is not new. It was proved  by J. Nat\'{a}rio \cite{N} via  an asymptotic analysis at infinity for the normal flow and an application of the isoperimetric inequality. The rigidity statement is not obtained in \cite{N}. 
\vspace{0.1cm}

\noindent \textbf{Outline.} The article is organized as follows. In Section \ref{sec:isomperimetricconvexcocompact} we define an isoperimetric problem for manifolds with outermost minimal surfaces. We present  basic properties and then describe the behavior of the Hawking mass function on convex co-compact hyperbolic manifolds that will be needed in Section \ref{sec:isoperimetriccomparison}. In Section \ref{sec:VR} we introduce renormalized volume, following the correspondence between equidistant foliations and metrics at the conformal infinity. In Sections \ref{sec:foliationatinfinity} and \ref{sec:uniquenessisomperimetric} we describe how boundaries of isoperimetric regions foliate the ends of convex co-compact hyperbolic manifolds. Section \ref{sec:isoperimetriccomparison} is where we prove one of the main results by relating renormalized volume to the asymptotic behavior of isoperimetric profiles. In Section \ref{sec:minkowskiinequality} we apply similar techniques to prove a Minkowski inequality. In the Appendix we describe the isoperimetric profile  for  the region between two geodesic planes and for cyclic quotients of $\mathbb{H}^3$. Although it is a parallel discussion from the article's main content, this was the starting of the authors collaboration and   we include for completeness sake. We show that geodesic spheres and tubes about geodesics are the only  stable constant mean curvature surfaces.

\section*{Acknowledgments} \noindent We are grateful to the Institute for Advanced Study where this work was initiated during the Special Year of Variational Methods in Geometry.
C.V. also thanks Richard Schoen and the University of California Irvine for the support and  hospitality during his stay at UCI. 
The authors would like to thank Ian Agol, Lucas Ambrozio, Martin Bridgeman, Jeff Brock and Andr\'{e} Neves for their comments and interest in this work.

\section{Isoperimetric regions in convex co-compact hyperbolic $3$-manifolds}\label{sec:isomperimetricconvexcocompact}

 \noindent  A complete hyperbolic $3$-manifold $M$ is called \textit{convex co-compact} if there exist a compact convex set $U$ such that the  exponential map from $\partial U$ to the conformal infinity $\partial M$ is a diffeomorphism. Each component of  $\partial M$ is assumed  to be incompressible  in $M$ and has negative Euler characteristic. In particular, $\partial M$ is always disconnected.
Moreover, we say  $M$   is \textit{acylindrical} if any map $(S^1\times[0,1],S^1\times\lbrace0,1\rbrace)\rightarrow (M,\partial M)$ is homotopic relative to $S^1\times\lbrace0,1\rbrace$ into $\partial M$. 

An important class of convex co-compact hyperbolic $3$-manifold are the quasi-Fuchsian metrics. A Fuchsian $3$-manifold is simply $\Sigma \times \mathbb{R}$ with the  metric $g=dr^2+ \cosh^2(r)g_{\Sigma}$ where  $\Sigma$ is a closed orientable hyperbolic surface of genus $g>1$. One can observe that the surface $\Sigma\times\lbrace0\rbrace$ is totally geodesic and that each boundary at infinity has a conformal structure that can be identified to $\Sigma$ (after reversing orientation for one of the ends). A quasi-Fuchsian $3$-manifold is a convex co-compact manifold of the form $\mathbb{H}^3/G$, where $G$ is a quasi-conformal deformation of a Fuchsian group. This corresponds to having potentially distinct conformal structures at the boundaries, with equality (after reversing the orientation of one component) if and only if the manifold is Fuchsian. If $M$ is a convex co-compact manifold with $\partial M$ incompressible, then the covering associated to each boundary is quasi-Fuchsian.  

For $M$ acylindrical hyperbolic $3$-manifold, there exists unique hyperbolic structure so that each end is Fuchsian (see for instance \cite[Corollary 4.3]{McMullen}). This will be the model hyperbolic structure in $M$ we will use to compare isoperimetric profiles for $M$ acylindrical (we will use Fuchsian structure as models for quasi-Fuchsian structures). To see the existence of the acylindrical hyperbolic structure with Fuchsian ends (following \cite{McMullen}) one uses that the map from the conformal boundary of $M$ to the opposite side of the boundary coverings (also referred as \emph{skinning map}) is a contraction, so one can use a fixed point argument to find the unique hyperbolic structure in $M$ where each end has matching conformal boundaries, hence Fuchsian. We will denote such manifold by $M_{TG}$. This notation comes from the fact that the convex core (smallest convex set containing all closet geodesics) of $M_{TG}$ has totally geodesic boundary.

\vspace{0.1cm}

    Let $\Omega_0$ be the largest volume compact region in $M$  with the property  that  the boundary   $\partial \Omega_0$ is a minimal surface which is  homologous and diffeomorphic $\partial M$. Such region exist since $\partial M$ is incompressible, convex  (see \cite{MSY}). Observe as well that there exists a unique minimal surface configuration when the metric has totally Fuchsian ends. In particular, $\partial \Omega_0$ is connected in each end of $M$.
     We call the surface $\partial \Omega_0$ the \textit{outermost minimal surface}. If $M$ contains  only one minimal surface, then $\Omega_0$ has zero volume.
 
  We will introduce the related isoperimetric problems which are relevant to  the discussion in the next sections. For each $V>0$ we  consider  
\begin{eqnarray*}
\mathcal{R}_V^1&=&\{\Omega: \Omega\subset M\,\, \text{is a compact region with}\,\, \Omega_0\subset \Omega\,\, \text{and}\,\, vol(\Omega-\Omega_0)=V  \} \\
\mathcal{R}_V^2&=&\{\Omega: \Omega\subset M\,\, \text{is a compact region with}\,\,  \text{and}\,\, vol(\Omega)=V   \}
\end{eqnarray*}
and let 
\begin{equation}\label{isoperimetric problem with horizon}
I_M(V)= \inf\{area(\partial\Omega):\,\,\Omega\in \mathcal{R}_V^1  \} \quad \text{and} \quad
J_M(V)= \inf\{area(\partial\Omega):\,\,\Omega\in \mathcal{R}_V^2  \}. 
\end{equation}

In order to distiguish both notations, we will refer to $I_M(V)$ as the \textit{outermost isoperimetric profile} of $M$, while we refer to $J_M(V)$ as the usual \textit{isoperimetric profile}.

\noindent  For any   sequence of points $x_i$ diverging to infinity,   the injective radius $inj_{x_i}M$ becomes unbounded. In particular, $M$ has bounded geometry. For each volume $V$ there exists a minimizing sequence for the isoperimetric problem (\ref{isoperimetric problem with horizon}) that does drift off to infinity. For the regularity near the outermost boundary $\partial \Omega_0$ see  \cite[Section 4]{EM}. In the following lemma we summarize well known existence and regularity results \cite{A,M,SS}.
\begin{lemma}
	For every $V>0$, there exists an isoperimetric region  $\Omega\in \mathcal{R}_V^i$, $i=1,2$,  with $vol(\Omega)=V$. The surface $\Gamma_{\Omega}=\partial \Omega$ is a volume preserving stable constant mean curvature surface. Namely,
	\begin{equation*}
	 \int_{\Gamma} |\nabla f|^2 - (-2 + |A|^2)f^2\, d\Gamma \geq 0 \quad \text{whenever}\quad \int_{\Gamma} f\,d\Gamma=0.
	\end{equation*}
\end{lemma}
\noindent Let us discuss the analytical properties of the isoperimetric profile.  Let $\Omega$ be an isoperimetric region in $M$ such that $vol(\Omega)=V$. Let $I$  denote both $I_M$ and $J_M$ in what follows. We first note that in our setting $I$ is absolutely continuous and twice differentiable
	almost everywhere, see \cite{FN}. In particular, the function $I(V)$ has left and right derivatives $I_-^{\prime}(V)$ and $I_+^{\prime}(V)$ and if $H$ is the  mean curvature (average of principal curvatures) of $\Gamma=\partial \Omega$ in the direction of the inward unit vector, then 
\begin{eqnarray} \label{first derivative}
(I)_{+}^{\prime}(v)\leq2H\leq (I)_{-}^{\prime}(v).
\end{eqnarray}
The second derivative exists weakly in the sense of comparison functions. More precisely, we say $f^{\prime\prime}\leq h$ weakly at $x_0$ if there exists a smooth function $g$ such that $f \leq g$, $f(x_0)=g(x_0)$, and $g^{\prime\prime}\leq h$. In this sense we have
\begin{eqnarray}\label{second derivative}
I(v)^2\, I^{\prime\prime}(v)+\int_{\Gamma}\big(Ric_g(N,N)+|A|^2\big)\,d\Gamma\leq\,0.
\end{eqnarray}
Let us sketch the proof of  (\ref{first derivative}) and (\ref{second derivative}):

Let $\Gamma_v$ be the variation $\Gamma_t= exp_{\Gamma}(tN)$   of $\Gamma$   re-parametrized  in terms of the enclosed volume $v(t)$ and let $\phi_0(t)$ (resp. $\phi_V(v)$) be the area of $\Gamma_t$ (resp. $\Gamma_v$).  Note that  $\phi_V(v)\geq I(v)$ and  $\phi_V(V)=I(V)$.
By the first variation formula for the  area and volume we have  $\phi_0^{\prime}(0)= 2\,H\, |\Gamma|$  and 
$v^{\prime}(0)=|\Gamma|$ respectively.
Since $\phi_0^{\prime}(t)=\phi_V^{\prime}(v)v^{\prime}(t)$, we conclude that
$\phi_V^{\prime}(v(0))= 2H$ and also that
$v^{\prime}(0)^2\phi_V^{\prime\prime}(v(0))=\phi_0^{\prime\prime}(0)-\phi_V^{\prime}(v(0))v^{\prime\prime}(0)$.
On the other hand,  it follows from the second derivative of area for general variations that
\begin{eqnarray*}
	\phi_0^{\prime\prime}(0)&=&
	-\int_{\Gamma}1\,L\,1\,d\Gamma +2H\,v^{\prime\prime}(0)\\
	&=& - \int_{\Gamma}\big(Ric_g(N,N)+|A|^2\big)\,d\Gamma +2H\, v^{\prime\prime}(0).
\end{eqnarray*}
Hence, in the sense of comparison functions,  (\ref{second derivative}) follows from:
\begin{eqnarray}\label{second derivative area}
\phi_V(v(0))^2\, \phi_V^{\prime\prime}(v(0))+\int_{\Gamma}\big(Ric_g(N,N)+|A|^2\big)\,d\Gamma=\,0.
\end{eqnarray}

\noindent One of the main reasons to require that the outermost minimal core is contained in each candidate region is so that the profile $I_M$ is monotone non-decreasing. We quickly justify this in the following lemma.

\begin{lemma}\label{profile increasing}
	The isoperimetric profile $I_M(V)$  is a strictly increasing function. 
\end{lemma}
\begin{proof}
Since $I_M$ satisfies (\ref{first derivative}) it is enough to show $I_M^{\prime}(V)>0$ at   volumes  $V$ for which $I_M^{\prime}(V)$ exists. Arguing by contradiction, we assume that  $I_M^{\prime}(V)<0$. It follows from (\ref{first derivative}) that $M-\Omega$ is mean convex. On the other hand,  the equidistant surfaces $\Sigma_t=exp_{\partial M}(tN)$ bounds convex regions for sufficiently large $t$. By Meeks-Simon-Yau \cite{MSY}, there  exists a compact minimal surface strictly between $\partial \Omega_0$ and $\Sigma_t$ which is homologous to $\partial M$. This contradicts the assumption that $\Omega_0$ is the outermost region.
\end{proof}

\noindent Let us now discard spherical and tori components as boundaries  of  isoperimetric regions for the outermost isoperimetric problem $I_M$. 
This will be useful to study the sign and monotonicity for the Hawking mass  below.

\begin{lemma}\label{components are hyperbolic}
	If $\Omega$ is an isoperimetric region with respect to $I_M$, then each connected component of $\partial \Omega$ has genus at least two. Moreover,  $\chi\big(\partial \Omega\big)\leq \chi(\partial M)$.
\end{lemma}
\begin{proof}
	If $\partial \Omega$ contains a spherical or a torus component, then such component is either  a geodesic sphere by Hopf's Theorem or  a tube about a geodesic by Ritor\'{e}-Ros \cite{RR}. In particular, $\partial \Omega$ is disconnected and its mean curvature satisfies $H>1$, see equation (\ref{mean curvature tubes}) in Subsection 6.3. Let us choose two components $\Gamma_1$ and $\Gamma_2$. We consider the function $f\in C^{\infty}(\Gamma)$, such that $f=|\Gamma_2|$ on $\Gamma_1$ and $f=-|\Gamma_1|$ on $\Gamma_2$ and $f=0$ otherwise. Hence, $\int_{\Gamma} fd_{\Gamma}=0$. By the stability inequality we have
	\begin{eqnarray*}
		0&\leq& -\int_{\Gamma}fLf\,d_{\Gamma} = - |\Gamma_2|^2\int_{\Gamma_1} (-2+ |A|^2) 
		- |\Gamma_1|^2 \int_{\Gamma_2} (-2+ |A|^2) \\ 
		&\leq& -2\bigg(|\Gamma_1|^2|\Gamma_2|+ |\Gamma_2|^2|\Gamma_1|\bigg)(-1+H^2)<0.
	\end{eqnarray*}
With this contradiction, we conclude that each component of $\partial \Omega$ has genus at least two. If we minimize area in the isotopy class of $\partial \Omega$ with respect to the model metric where the outermost region is totally geodesic following \cite{MSY}, then we obtain the unique connected   minimal surface $S$ with multiplicity $m$ such that $ g(\partial M)=mg(S)\leq g(\partial \Omega)$. Since each component of $\partial \Omega$ has genus greater than one, we obtain $\chi(\partial M)\geq \chi(\partial \Omega)$.
\end{proof}

\noindent In order to study the behavior of the isoperimetric profile, it is convenient to define the Hawking mass function, see \cite{CESZ,JSZ,LN}.
 
 \begin{definition}
 	The Hawking mass function $m_H: (0,,\infty)\rightarrow \mathbb{R}$ is defined in terms of the outermost  isoperimetric profile $I_M$:
 	\begin{equation}\label{hawking mass}
 	m_H(V)= \sqrt{I_M(V)}\bigg(2\pi\chi(\partial M) + I_M(V) - \frac{1}{4}I_{M}^{\prime +}(V)^2\, I_M(V)\bigg).
 	\end{equation}
 \end{definition}

\noindent The quantity $m_H(V)$ is monotone and has a sign as  shown below. Theses facts will be useful in the description and comparison for  the isoperimetric profile in Section 4.

\begin{lemma}\label{monotonicity}
	The Hawking mass function $m_H(V)$  is a monotone non-decreasing function.
\end{lemma}
\begin{proof}	
	Let $\Omega$ be an isoperimetric region of volume $V$  and  $\Gamma=\partial \Omega$.
	Recall the Gauss equation 
		\begin{eqnarray}\label{Gauss eq}
	Ric(N)+|A|^2= \frac{R}{2} -K +3H^2 +\frac{|\mathring{A}|^2}{2}.
	\end{eqnarray}
	 Using (\ref{Gauss eq}) we can rewrite inequality (\ref{second derivative area}) as follows
	\begin{eqnarray}\label{second derivative profile}
	I_M^{\prime\prime}(V)\leq \phi_V^{\prime\prime}(V)&\leq & \frac{1}{\phi_V(V)^2}\int_{\Gamma}\bigg( 3 +K -3H^2 -\frac{|\mathring{A}|^2}{2}\bigg) d_{\Gamma} \nonumber \\
	&\leq&\frac{\sum_{i=1}^{l}4\pi(1-g(\Gamma_i))}{I(V)^2} + \frac{3}{I(V)} - \frac{3I^{\prime}(V)^2}{4I(V)}.
	\end{eqnarray}
	The last inequality follows from the Gauss-Bonnet Theorem applied to each component of $\Gamma$. Next we compute $m_H^{\prime}$ in the distribution sense  
	  following similar computation in  \cite[Lemma 3]{E1}: 
	  
	  For this we need the quotient difference operator $\Delta_{\delta}f(V)=\delta^{-1}\big(f(V+\delta)-f(v)\big)$, $\delta\neq 0$. Let $\varphi\in C^1_0{(0,\infty)}$, then	\begin{eqnarray*}
	  -	\int \varphi^{\prime}\,m_H &=&  -\lim_{\delta \rightarrow 0}\int \varphi^{\prime} \sqrt{I_M}\bigg(2\pi\chi(\partial M) + I_M - \frac{1}{4}(\Delta_{\delta}I_M)^2\, I_M\bigg)\\
	  -	\int \varphi^{\prime}\,m_H &=&  \lim_{\delta \rightarrow 0}\int \varphi \Delta_{-\delta} \bigg( \sqrt{I_M}\big(2\pi\chi(\partial M) + I_M - \frac{1}{4}(\Delta_{\delta}I_M)^2\, I_M\big)\bigg) \\
	  		-\int \varphi^{\prime}\,m_H&=& \lim_{\delta \rightarrow 0}\int \varphi\bigg(I_M^{\prime} -  \frac{1}{2}I_M\,I_M^{\prime}\,\Delta_{-\delta}(\Delta_{\delta} I_M)- \frac{(I_M^{\prime})^2}{4}I_M^{\prime} \bigg)\sqrt{I_M} \\
	  			&&\,\,\,+\,\,\,\varphi \bigg(\pi\chi(\partial M)+ \frac{I_M}{2} - \frac{(I_M^{\prime})^2}{8} \,I_M  \bigg)\frac{I_M^{\prime}}{\sqrt{I_M}}\\
		&=&  \int \varphi \frac{I_M^{\prime}\,I_M^{\frac{3}{2}}}{2}\bigg(\frac{(I_M^{\prime})^2}{4I_M} + \frac{2\pi\chi(\partial M)}{I^2} \bigg) - \lim_{\delta \rightarrow 0} \int \varphi\frac{I_M^{\prime}\,I_M^{\frac{3}{2}}}{2} \Delta_{-\delta}(\Delta_{\delta}I_M).
	\end{eqnarray*}
we used in above formulas that $I_M^{\prime+}=I_M^{\prime-}$ except at possibly countable many points. Using that $\limsup_{\delta \rightarrow 0}\Delta_{-\delta}\Delta_{\delta}I_M\leq \phi_V^{\prime\prime}(V)$ and
	applying inequality \ref{second derivative profile}, we obtain:
	\[
-\int \varphi^{\prime}\,	m_H^{\prime} \geq \int \varphi \frac{I_M^{\prime}(V)\,I_M^{\frac{3}{2}}(V)}{I_M^2(V)} \bigg(\pi\chi(\partial M) - \sum_{i=1}^{l_V}2\pi(1-g(\Gamma_i)) \bigg)\geq0,
	\]
by  Lemma \ref{profile increasing} and Lemma \ref{components are hyperbolic}. Hence, $m_H^{\prime}(V)\geq 0$ in the distribution sense. It follows from (\ref{first derivative}) that at every discontinuity point of $I_M^{\prime}(V)$, the Hawking mass $m_H(V)$ jumps up. Therefore, $m_H(V)$ is monotone non-decreasing.	
\end{proof}

\begin{remark}
If  $M_e$ is a connected component of $M-\Omega_0$, then  the isoperimetric problem for this end is defined as follows. In the class $\mathcal{R}_V^0=\{\Omega: \Omega\subset M_e\,\, \text{is a compact region with}\,\, \partial M_e\subset \Omega\,\, \text{and}\,\, vol(\Omega)=V  \}$,  we set the  isoperimetric profile $I_e$ as
\[
I_e(V)= \inf\{area(\partial\Omega):\,\,\Omega\in \mathcal{R}_V^0 \}\,\,-\,\, area(\partial M_e)
\]
The results in this section  and their proofs  extended naturally to the   profile $I_e(V)$. The  only relevant change needed is in the definition of the Hawking mass $m_H(V)$.
\end{remark}

\section{Renormalized volume}	\label{sec:VR}
The convex co-compact hyperbolic $3$-manifolds $(M,g)$ are particular examples of conformally compact manifolds. Namely, there exists  on a compact manifold with boundary $\overline{M}$ a defining function $x$ (i.e., $x>0$ on  $ M$, $x=0$ on $\partial M$, and $dx\neq 0$ on $\partial M$)  such that $M=int(\overline{M})$ and the conformal metric  $\overline{g}=x^2g$ extends smoothly to the boundary $\partial M$. The restriction of $\overline{g}$ to $\partial M$ defines a well defined conformal structure $[\partial M]$ on $\partial M$ which we will refer to as the conformal infinity (or conformal boundary) of $M$.

By  works of Epstein \cite{E} and Graham \cite{G},  for each metric $h$ on the conformal infinite $[\partial M]$ there exist a unique  defining function $x$ defined in a collar neighborhood of $\partial M$ with the special property that the level sets $\{x=r\}$ yield an equidistant foliation of the ends of $M$ by convex sets. Moreover,
\[
g=\frac{1}{x^2}\,\big(dx^2 + h_0 + h_2 x^2 + h_3 x^3 + \ldots\big),
\]
where $h_i$ are tensors in $\partial M$.
We now look for the quantity $Vol(\{x>\varepsilon\})$ as $\varepsilon$ approaches zero. One can check that
\[
Vol(\{x>\varepsilon\})=  c_0\, \varepsilon^2 - L \log(\varepsilon) + V + o(1).
\]
The constants $c_0$ e $L$  depend only on the metric $h\in[\partial M]$ uniquely associated to the equidistant foliation $\{x=\varepsilon\}$. The quantity $V$ is the renormalized volume associated to the metric $h$.

In the context of hyperbolic metrics, the work \cite{KS} provided a renormalization procedure for computing $V$ in terms of the geometric data associated to equidistant foliation was proposed. Let  $M_r$, $r\geq 0$, be an equidistant  foliation at infinity of $M$   by convex sets.  This foliation induces a Riemannian metric at the conformal infinity $\partial M$ as follows:
\begin{eqnarray}\label{metric at infinity}
h=\lim_{r\rightarrow \infty}\, e^{-2r} g_r,
\end{eqnarray}
where $g_r$ is the induced Riemannian metric on $\partial M_r$. Using the unique correspondence between metrics in the conformal class $[\partial M]$ and equidistant foliations as above \cite{G}, we define the\textit{W-volume} with respect to $h$ as
\begin{equation}\label{RenVol}
W(M,h) := \text{vol}(M_r)- \frac{1}{2} \int_{\partial M_r} H\, da + r\pi \mathcal{X}(\partial M).
\end{equation}
One can show that  the value on the right hand side is independent of $r$, see \cite{KS}. Among all metrics $h\in [\partial M]$ of fixed area, say $A_0=Area(\partial M, h)$, the W-volume is maximized by the unique hyperbolic metric $h_0$ having that area, see  Proposition \ref{polyakov} in the Appendix.  This motivates the following:
\begin{definition}The Renormalized volume  of $M$ is defined as $V_R(M)= W(M, h_{hyp})$ where $h_{hyp}$ is the hyperbolic metric of Gauss curvature $-4$.
	\begin{example}[Renormalized volume of Fuchsian manifolds]
Let $(M,g)$ be a Fuchsian $3$-manifold. Recall that $M=\Sigma\times \mathbb{R}$ where $\Sigma$ is a closed orientable surface of genus $g$ and \[
g= dr^2 + \cosh^2(r)g_{\Sigma}
\]
where $g_{\Sigma}$ is the hyperbolic metric (Gauss curvature $-1$) on $\Sigma$. With the warped metric $g$ one can check that  the slices $\Sigma_r$ form a global equidistant foliation of $\Sigma\times \mathbb{R}$ by totally umbilical surfaces. A simple computation gives that the metric at infinity via the limit procedure (\ref{metric at infinity}) is the hyperbolic metric $h_{hyp}$ of Gauss curvature $-4$. Because the W-volume $W(\Sigma\times \mathbb{R},g_{hyp})$ does not depend of the choice of slice $\Sigma_r$, we obtain that $V_r(M)=0$ by choosing $M_r=\Sigma\times[-r,r]$ with $r$ sufficiently closed to zero in (\ref{RenVol}).
	\end{example}
\end{definition}The following theorem studies the minimum of $V_R$ along a moduli class of hyperbolic $3$-manifolds.
\begin{theorem} [Bridgeman-Brock-Bromberg \cite{BBB}, Vargas Pallete \cite{P}] \label{Comparison Renormalized volume}
	Let  $M$ be a convex co-compact hyperbolic $3$-manifold. Then $V_R(M)\geq \frac{v_3}{2}\Vert DM \Vert$, where $v_3$ is the volume of a regular ideal tetrahedra, $\Vert.\Vert$ denotes the Gromov norm of a 3-manifold, and $DM$ is the double of $M$. Moreover, equality $V_R(M) = \frac{v_3}{2}\Vert DM \Vert$ occurs if and only if the convex core of $M$ has totally geodesic boundary, in which case $M$ needs to be either acylindrical with Fuchsian ends or Fuchsian.
\end{theorem}

 \section{Foliation at infinity for hyperbolic ends}\label{sec:foliationatinfinity}

\noindent As discussed above,  for any  metric $h$ in the conformal boundary $\partial M$ there exists the unique special    defining function $x$ whose level sets form a equidistant foliation of the ends of $M$ and we write \[g=\frac{1}{x^2}\,\big(dx^2 + h_0 + h_2 x^2 + h_3 x^3 + \ldots\big),\] 
where  $h_{j}$ are tensors on $\partial M$. It is known that the tensor $h_2$ is formally undetermined but it satisfies $tr_{h_0}h_2=K$ where $K$ is the scalar curvature of $h_0$. The mean curvature of $\Sigma_r$ satisfies
 \[
 H_r= 1 + \frac{tr_{h_0}h_2}{4}\, r^2 + o(r^2).
 \]
 In particular, if $h_0$ has negative Gauss curvature, which is always possible since $\chi(\partial M)<0$ at each component, then  $H_r$ is almost constant and  $H_r<1$ when $r$ is sufficiently close to $0$.

It is possible to perturb these level sets  to  to  have constant mean curvature. A direct application of the main result of \cite{MP} gives:

 \begin{theorem}[Mazzeo-Pacard \cite{MP}]
 	Let $M$ be a convex co-compact hyperbolic $3$-manifold with $\chi(\partial M)<0$. There exists a compact subset $K$ such that $M-K$ has an unique monotone increasing  foliation by  constant mean curvature surfaces.
 \end{theorem}

Given that the mean curvature is monotone increasing along the foliation, we obtain the following corollary concerning CMC surfaces in the foliated region.

\begin{corollary}\label{uniqueness}
	If $\Sigma$ is an connected constant mean curvature surface embedded in $M-K$  and homologous to $\partial (M-K)$, then $\Sigma$ is a leaf of the canonical foliation.
\end{corollary}
\begin{proof}
 Note that $\Sigma$ is tangent to the outer radius  leaf $\partial \Omega_A$ and the inner radius leaf $\partial\Omega_B$  of the canonical foliation with mean curvature vectors having the same orientations. By the Maximum Principle,  $H_A\leq H\leq H_B$. Since the foliaton's mean curvature  is strictly increasing for large volumes, we have  that $\Sigma=\partial \Omega_A=\partial \Omega_B$.
\end{proof}

 \noindent The next two results strengthen the variational characterization of each leaf of the canonical foliation. Corollary \ref{H=1} will be used in the proof of Theorem \ref{foliation is isoperimetric}.

\begin{proposition}\label{Brane Functional}
There exist an embedded strongly stable constant mean curvature closed surface $\Sigma$ with mean curvature $H$ for each $H\in(0,1)$. 
\end{proposition}
\begin{proof}
	We consider sets $\Omega$ such that $\partial\Omega=\Gamma$ is homologous to $\partial M$. In this class, let $\mathcal{F}_H$ be the brane action functional $\mathcal{F}_H(\Omega)=|\Omega|- \frac{1}{2H}|\Gamma|$, see \cite{ACG}. Recall that the volume element of $M$ satisfies $dM=d\Lambda$ for some $n-1$ form $\Lambda$. In particular, $\mathcal{F}_H$ is a functional of $\Gamma$  and $\mathcal{F}_H(\Gamma)=\int_{\Gamma} \Lambda  - \frac{1}{2H}A(\Gamma)$ where $A(\Gamma)$ is the area of $\Gamma$. Note  the form of $\mathcal{F}_H$ is to agree with the  inequality $V-\frac{1}{2}A>0$ for large volumes $V$, see  Lemma \ref{cheeger} below.  We remark that we do not have a sign for $\mathcal{F}_H$. If  $\Omega_r$ is such that  $\Gamma_r$ is a connected CMC leaf of the canonical foliation far out at infinity, then $\Gamma_r$ is a barrier for $\mathcal{F}_H$ in that sense that  \begin{eqnarray*}
		\delta \mathcal{F}_H(\Omega_r)(fN)= \int_{\Gamma}(-1+\frac{H_r}{H})f\,d\Gamma > 0\end{eqnarray*}
	where $H_r$ is the mean curvature of $\Gamma_r$ with respect to  the inward unit normal vector $N$ of $\Omega_r$ and $f$ is a positive function. In other words, the functional $\mathcal{F}_H$ decreases as one approaches $\partial M$ (in fact, $\lim_{r\rightarrow \infty} \mathcal{F}_r(\Omega_r)=-\infty$.)
	We consider the  maximization problem $\sup \{	\mathcal{F}_H(\Omega)\}$,  $\Omega$  as described above,  inside the compact set $\Omega_r$.
	By  maximization arguments using barriers, see \cite{MSY},   there exist a    maximizer $\Omega_H$ for the functional $\mathcal{F}_H$ which does not intersect $\partial \Omega_r$. We need to show that $\Omega$ is non-trivial, i.e., $\Omega$ has non-zero volume: this only happens if $\Omega=S$ where $S$ us a surface homologous to one component of $\partial M$. In particular, it must be the homological non-trivial surface of least area in $M$. On the other hand these surfaces are also barrier for maximizing the functional $\mathcal{F}_H$. Indeed, \begin{eqnarray*}
			\delta \mathcal{F}_H(\Omega)(fN)= \int_{\Gamma}(1-\frac{0}{H})f\,dS >0 \end{eqnarray*} where $f$ is a positive function on $S$ and $N$ is the outward unit normal vector to $\Omega$. Moreover, the first and second variation for  $\Gamma_{\Omega_H}$, see \cite{ACG}, implies it has constant mean curvature $H$ and
	\begin{eqnarray}\label{brane stability}
	\delta^2 \mathcal{F}(\Omega_H)(f,f)=-\frac{1}{2H}\int_{\Gamma}|\nabla f|^2 + (2-|A|^2)f^2\,d_{\Gamma_H}\leq 0,
	\end{eqnarray}
for all function $f\in C^{\infty}(\Gamma)$.	In other words, the surface $\Gamma_H$ is a strongly stable constant mean curvature surface.
\end{proof}
This construction can be made in each end of $M-\Omega_0$, where $\Omega_0$ is the outermost region.
\begin{corollary}\label{H=1}
	There are no embedded cmc  surface $\Sigma\subset M-\Omega_0$   homologous to $\partial M$ with constant mean curvature $H=1$. 
\end{corollary}
\begin{proof}
Let $\Omega$ be the region bounded by  $\Sigma$.   Following Proposition \ref{Brane Functional}, we can use $\Sigma$ as a barrier for maximizing  the brane action functional $\mathcal{F}_{H=1}$ inside $\Omega$ among competitors homologous to $\Sigma$. Note that the maximizer $\Gamma$ is non-trivial by the same argument in Proposition \ref{Brane Functional} (in other words, $\Gamma$ must enclose some volume.) Consequently, we obtain  that either $\Sigma$ is strongly stable (in the sense (\ref{brane stability})) or we can replace $\Sigma$  by a strongly stable compact  surface $\Gamma$ with constant mean curvature $H=1$.  Since $H=1$, the Jacobi operator becomes $L=\Delta + |\mathring{A}|^2$. Therefore, applying the test function $f=1$ in the stability inequality (\ref{brane stability}) yields a contradiction. 
\end{proof}

\section{Uniqueness of isoperimetric regions}\label{sec:uniquenessisomperimetric}

We start with a simple lemma providing an useful inequality between the enclosed volume and area of equidistant surfaces.

\begin{lemma}\label{cheeger}
	Let $\Omega$ be a strongly convex set in a convex co-compact hyperbolic $3$-manifold $M$ such that $\partial \Omega$ is homologous to $\partial M$.  Then $\text{Area}(\partial \Omega_r)<2\,vol(\Omega_r)$, where $\Omega_r=\{x\in M: d(x,\Omega)\leq r\}$ and $r$ sufficiently large.
\end{lemma}

\begin{proof}
Let $\Sigma_r=\partial \Omega_r$. The induced metric on the level set $\Sigma_r$ is
\[
g_r=g_0\big(\cosh(r)I  + \sinh(r)A\,,\,\cosh(r)I  + \sinh(r)A  \big)
\]
where $I$ and $A$ is the identity and the second fundamental form of $\Sigma_0$. The metric is well define since $\Sigma_0$ is strongly convex. Hence,
\begin{eqnarray*}
|\Sigma_r|&=& \int_{\Sigma} \det \big( \cosh(r)I  + \sinh(r)A \big)\, d\Sigma \\
&=& \int_{\Sigma} \bigg(\cosh(2r)+ \sinh(2r)\,H + K \big(\frac{\cosh(2r)-1}{2}\big) \bigg)\, d\Sigma.
\end{eqnarray*}
On the other hand, 
\begin{eqnarray*}
	|\Omega_r|&=& |\Omega| + \int_{0}^{r}\int_{\Sigma} \bigg(\cosh(2\rho)+ \sinh(2\rho)\,H + K \big(\frac{\cosh(2\rho)-1}{2}\big) \bigg)\, d\Sigma\, d\rho \\
	&=& |\Omega|  + \frac{1}{2}\int_{\Sigma} \bigg(\sinh(2r) + \big(\cosh(2r) -1\big)\,H + K\, \frac{\sinh(2r)-2r}{2}\bigg) d\Sigma
\end{eqnarray*}
Hence,
\begin{eqnarray*}
	2|\Omega_r|- |\Sigma_r|= 2|\Omega| &+& \int_{\Sigma} (\cosh(2r)-\sinh(2r))(H-1) \, d\Sigma  - \int_{\Sigma} H\, d\Sigma\\
	&+& \frac{\pi}{2}\mathcal{X}(\Sigma) (\sinh(2r)- \cosh(2r))  - \pi\mathcal{X}(\Sigma) (2r-1).
\end{eqnarray*}
 Therefore, $\lim_{r\rightarrow \infty} (2|\Omega_r|-|\Sigma_r|)=+\infty$ since  $\mathcal{X}(\Sigma)<0$.
\end{proof}

We are now ready to show that leaves of the cmc foliation at infinity are in fact isoperimetric, for leaves sufficiently deep into the end. We will use this result in the next section for the isoperimetric comparison results and characterization of the Renormalized Volume.

\begin{theorem}\label{foliation is isoperimetric}
	Let $M$ be a convex co-compact hyperbolic $3$-manifold and $\{\Sigma_H\}_{H\in \mathbb{R}}$  the  cmc foliation at infinity parametrized by  constant mean curvature  $H$. If $H$  is sufficiently close to one, then $\Sigma_H $ is uniquely isoperimetric (with respect to either $I_M$ or $J_M$) for the volume it encloses in $M$.
\end{theorem}
\begin{proof}
Let $\Omega_{V_i}$ be an isoperimetric region of volume $V_i$.	Let us study first the case $\partial\Omega_{V_i}\cap K\neq \emptyset$, where $K\subset M$ is some fixed compact set and $\lbrace V_i\rbrace$ is a sequence of volumes satisfying $V_i\rightarrow \infty$. By compactness theorem for isoperimetric  surfaces, $\partial \Omega_{V_i}$ converges in the graphical sense and with multiplicity one to a non-compact stable constant mean curvature surface $\Sigma_{\infty}$ in  $M$.  Note that the mean curvature of $\Sigma_{\infty}$ satisfies $H\geq 1$. Indeed, by the maximum principle the mean curvature  of $\partial\Omega_{V_i}$ is greater than the mean curvature $H_{R_i}$ of $\Sigma_{R_i}$, where $\Sigma_{R_i}$ is a leaf of the cmc foliation  that encloses $\Omega_{V_i}$  and it is tangent to $\partial\Omega_{V_i}$. On the other hand, $\lim_{i\rightarrow \infty}H_{R_i}=1$. It follows from $H\geq 1$ that  the operator $P=\Delta+ |\mathring{A}|^2$ satisfies 
	\[
	0\leq -\int_{\Sigma_{\infty}} fP(f)\,d_{\Sigma_{\infty}}\quad \text{for every}\quad \int_{\Sigma_{\infty}} f\,d_{\Sigma_{\infty}}=0.
	\]
By the monotonicity formula, $\Sigma_{\infty}$ is either compact or  has infinite area. We will deal with the latter case first. In that case,  $\Sigma_{\infty}$  is also conformally equivalent to a closed Riemann surface with finite points removed by Fisher-Colbrie \cite{FC}. Theorem 1.6 in Da Silveira \cite{DS} applied to the operator $P=\Delta + |\mathring{A}|^2$ implies that $\mathring{A}\equiv 0$ and $\Sigma_{\infty}$ is totally umbilical. As $\Sigma_{\infty}$ is a non-compact surface with mean curvature $H\geq 1$, then it has to be an embedded oriented horosphere $H_0$ in $M$. Similarly, by the strong compactness properties for sequences of isoperimetric surfaces, any sequence of basepoints in $\partial\Omega_{V_i}$ will locally converge (after a subsequence) to a horosphere in $M$ or in $\mathbb{H}^3$ under the Cheeger-Gromov convergence for manifolds. And since horospheres are strictly convex, we have that $\partial\Omega_{V_i}$ will be locally convex for $i$ large enough. 

In general, isoperimetric regions might not be connected. Nevertheless,  each connected component is  also isoperimetric. By our assumptions, at least one component will have large volume and passing through a compact region, say $\Omega_{V_i}$. Following the prove of \cite[Theorem 3.3]{U}, we can use the local convexity of $\partial\Omega_{V_i}$ and the normal geodesic flow to see that the covering of $M$ associated to $\Omega_{V_i}$ (with the induced hyperbolic metric) is obtained by gluing $\Omega_{V_i}$ to $\partial\Omega_{V_i}\times\mathbb{R}^+_0$, where  $\partial\Omega_{V_i}\times\mathbb{R}^+_0$ has a metric so that $\partial\Omega_{V_i}\times\lbrace t\rbrace$ is locally convex for any $t\geq0$. We denote this covering by $\tilde{M}_{V_i}$, and it follows from the normal geodesic flow construction that the surfaces $\partial\Omega_{V_i}\times\lbrace t\rbrace$ are equidistant to one another. From this description of $\tilde{M}_{V_i}$ it follows that $\Omega_{V_i}$ is \textit{geodesically convex} in $\tilde{M}_{V_i}$, meaning that any geodesic in $\tilde{M}_{V_i}$ with endpoints in $\Omega_{V_i}$ is contained in $\Omega_{V_i}$. Since $\tilde{M}_{V_i}$ is the covering associated to $i_*:\pi_1(\Omega_{V_i})\rightarrow\pi_1(M)$, it follows that any homotopically trivial geodesic segment $\gamma:([0,1],\lbrace0,1\rbrace)\rightarrow(M,\Omega_{V_i})$ (i.e. homotopic into $\Omega_{V_i}$ relative to $\Omega_{V_i}$) has image in $\Omega_{V_i}$. We will use this to show in fact that $i_*:\pi_1(\Omega_{V_i})\rightarrow\pi_1(M)$ has trivial image and hence $\tilde{M}_{V_i}$ is $\mathbb{H}^3$.

Assume that $g\in\pi_1(M)$ is a nontrivial element in the image of $i_*:\pi_1(\Omega_{V_i})\rightarrow\pi_1(M)$. Hence we have a homotopically trivial geodesic segment $\gamma$ of $(M,\Omega_{V_i})$ which, for $i$ sufficiently large, lies close to the orthogeodesic of $\Sigma_\infty=H_0$ associated to $g$. But then such homotopically trivial geodesic segment $\gamma$ will not be contained in $\Omega_{V_i}$, which is a contradiction.

Since $\Omega_{V_i}$ lifts to $\mathbb{H}^3$, then it must be a hyperbolic geodesic ball. This is impossible for the outermost isoperimetric profile $I_M$, as $\Omega_{V_i}$ should contain $\Omega_0$. For the isoperimetric profile $J_M$, we saw on Lemma \ref{cheeger} we can construct sets $\Omega_r$ so that $\lim_{r\rightarrow \infty} (2|\Omega_r|-|\partial\Omega_r|)=+\infty$. Such competitors will beat hyperbolic geodesic balls for sufficiently large volumes, so we have a contradiction for the profile $J_M$ as well.

Now we deal with the case when $\Sigma_{\infty}$ is compact. In this case $\partial \Omega_V$ is disconnected which implies  by the proof of Lemma \ref{components are hyperbolic} that $H_{V_i}<1$ for large $i$. Hence,  $\Sigma_{\infty}$ has mean curvature $H=1$. This contradicts Corollary \ref{H=1} and, hence, $\partial\Omega_V$ diverges to infinite as $V\rightarrow \infty$. 
\vspace{0.1cm}
	
\noindent	Let us assume now  that the outermost region $\Omega_0\subset \Omega_{V_i}$ and that  $\partial \Omega_{V_i}$ is drifting towards  infinity. Since $\partial \Omega_{V_i}$ is homologous to $\partial \Omega_0$,   Corollary \ref{uniqueness} implies at least one component of $\partial \Omega_{V_i}$ in a fixed end is a leaf of the foliation. But this implies $\Omega_{V_i}$ is connected since other components would have larger mean curvature. Hence, each component of  $\partial \Omega_{V_i}$ is a leaf of the canonical foliation.  The argument so far shows that  the leafs of the canonical foliation are uniquely isoperimetric with respect to the outermost isoperimetric profile $I_M$ (the characterization also  holds  at  any   end of $M$).    Next we assume that $\Omega_0\subset \Omega_V^c$.  Let us show that such configuration is not  isoperimetric. Note that all  components must be  drifting off to infinity.
By previous argument, we can deduce that $|\partial \Omega_{V_i}|> |\Sigma_{R_i}|-|\Sigma_{R_j}|$, where the enclosed volume of the leaf $\Sigma_{R_j}$ is $V_j$ and the enclosed volume  of $\Sigma_{R_i}$  is $V_i+V_j$. By the Fundamental Theorem of Calculus, $|\Sigma_{R_i}|-|\Sigma_{R_j}|=  2H_{s_0}\,V_i$, where $s_0\in (R_j,R_i)$. In particular, $|\Sigma_{R_i}|-|\Sigma_{R_j}|> 2H_{R_j}\,V_i$. Hence, \[|\partial \Omega_{V_i}|>  |\Sigma_{R_j}|\, \frac{2\,H_{R_j}\,V_i}{|\Sigma_{R_j}|}.\]
	Now choose $R_j$ such that $V_j= V_i$. One can check that the profile $I_M(V)$ associated to the foliation  $\Sigma_H$ satisfies  $2H_V=I^{\prime}_M(V)>\frac{I_M(V)}{V}$ for large $V$. Indeed, this is equivalent showing that $\ln\big( \frac{I_M(V)}{V}\big)$ is an increasing function for large volume $V$.  For this just notice that $\lim_{V\rightarrow \infty} \frac{I_M(V)}{V}=2$, that  $|\Sigma_{R_k}|<|\Gamma_r|$, where $\Gamma_r$ is equidistant to a fixed cmc leaf $\Gamma_0$ and encloses the same volume as $\Sigma_{R_k}$,  and that $|\Gamma_r|<2\,V_k$ by Lemma \ref{cheeger}.  Therefore, $|\partial \Omega_{V_i}|> |\Sigma_{R_j}|$.
\end{proof}
\begin{remark}Alternatively, we could have worked with  the minimizers $\Sigma_H$ of $\mathcal{F}_H$, so that if we follow the same  steps in the proof above we would obtain that for $H$ sufficiently close to $1$ that $\Sigma_H$ is equal to the cmc leafs of the canonical foliation. Standard comparison implies that minimizers of $\mathcal{F}$ are isoperimetric with respect to the outermost profile $I_M$. As a result, the leafs of the canonical foliation are strongly stable in the sense (\ref{brane stability}).
	\end{remark}

Now we  strengthen Lemma \ref{monotonicity} by showing that the Hawking mass is non-positive.

\begin{lemma}\label{hawking mass negative}
	The Hawking mass  satisfies  $m_H(V)< 0$ for every $V$ unless the ends of $M$ are Fuchsian where  $m_H(V)\equiv 0$. 
\end{lemma}
\begin{proof}
	\noindent If $V$ is sufficiently large, then $\Gamma_V$ is connected on each end by Theorem \ref{foliation is isoperimetric} and have the topology of $\partial M$. Note also that by Theorem \ref{foliation is isoperimetric}, the isoperimetric profile $I_M(V)$ is differentiable for $V$ sufficiently large. Integrating the Gauss equation  and applying the Gauss-Bonnet Theorem on each component of $\Gamma_V$ gives \begin{eqnarray*}
		2\pi \chi(\partial M)&=&\int_{\Gamma_V} K_V\, d\Gamma= \int_{\Gamma_V} ( -1+ \det(A_V))\,d_{\Gamma}=\int_{\Gamma_V}( -1  +H_{V}^2 - \frac{1}{2}|\mathring{A}_V|^2 \,d\Gamma \\
		&\leq& \int_{\Gamma_V}(-1 + H_V^2)\,d\Gamma= -I_M(V)+\frac{1}{4}I_M^{\prime}(V)^2I_M(V).
	\end{eqnarray*}
	Therefore, $m_H(V)\leq 0$. Since $m_H(V)$ is non-decreasing, we conclude that $m_H(V)\leq 0$ for all volumes. If $m_H(V)=0$, then $\Gamma_V$ is totally umbilical and each end of $M$ is Fuchsian.
\end{proof}

\section{Isoperimetric profile comparison}\label{sec:isoperimetriccomparison}

\noindent Since our goal is to relate  $V_R$ to the isoperimetic profile, we start with an expression of $V_R$ in terms of volume and area only. We  achieve this by using the fact that for an equidistant foliation $M_r$, the mean curvature $H$ approaches $1$ exponentially on $r$. We recall that the equidistant foliation $M_r$ corresponds to the hyperbolic metric at infinity of Gauss curvature $-4$.

\begin{proposition}\label{prop:asympprofileVR}
	\begin{equation}
	\lim_{r\rightarrow\infty} Vol(M_r) - \frac12 Area (\partial M_r) + \pi\chi(\partial M)\log \sqrt{\frac{2Area(\partial M_r)}{\pi|\chi(\partial M)|}} = V_R(M) +\frac{\pi}{2}\chi(\partial M)
	\end{equation}
\end{proposition}
\begin{proof}
Let us first prove  that \begin{equation}\label{L1 norm H-1}
\lim_{r\rightarrow \infty} \int_{\partial M_r} (H-1)\,da  = \pi\chi(\partial M).
\end{equation}
	Fix $s$ and denote the metric at $\partial M_{s}$ by $g$, shape operator $A$ and its principal curvatures by $k_{1,2}$. Then the metric and principal curvatures at $\partial M_r$ are given by (see \cite{U})
	
	\begin{equation*}
	g_r(u,v) = g(\cosh(r-s)u+\sinh(r-s)Au, \cosh(r-s)v+\sinh(r-s)Av) 
	\end{equation*}
	
	\begin{equation*}
	k_{1,2}^r = \frac{\sinh(r-s) + \cosh(r-s)k_{1,2}}{\cosh(r-s) + \sinh(r-s)k_{1,2}}
	\end{equation*}
	In particular, the volume element at $\partial M_r$ is given by
	
	\begin{equation*}
	da_r = (\cosh(r-s) + \sinh(r-s)k_1)(\cosh(r-s) + \sinh(r-s)k_2) da_s
	\end{equation*}
	
	So then
	
	\begin{equation*}
	\int_{\partial M_r} (H_r-1)da_r = \frac{1}{4}\int_{\Sigma_s} (-1+k_1)(1+k_2)+(-1+k_2)(1+k_1) da_s+ O(e^{s-r}) 
	\end{equation*}
	
	Applying  the Gauss equation and Gauss-Bonnet Theorem:
	
	\begin{equation*}\label{eq:HvsArea}
	\lim_{r\rightarrow\infty}\int_{\partial M_r}(H_r-1)da_r = \frac{1}{2}\int_{\partial M_s}(-1+k_1k_2)da_s =\pi\chi(\partial M_s) 
	\end{equation*}
		The equidistant foliation is parametrized such that 
	$	\lim_{r\rightarrow\infty} e^{-2r}g_r = h_{0}$,
	where $h_{0}$ is the hyperbolic metric at infinity of Gauss curvature $-4$. In particular,
	
	\begin{equation*}
	\lim_{r\rightarrow\infty} e^{-2r}Area(\partial M_r) = \frac{\pi}{2}|\chi (\partial M)|
	\end{equation*} 	
	This is equivalent to	
	\begin{equation}\label{asymptotics r}
	\lim_{r\rightarrow\infty} \left(r-\log\sqrt{\frac{2Area(\partial M_r)}{\pi\chi(\partial M)}}\,\right) = 0
	\end{equation}
If we  add and subtract both  terms $\frac{1}{2}\,\text{Area}(\partial M_r)$ and $\log\sqrt{\frac{2Area(\partial M_r)}{\pi\chi(\partial M)}}$  in  the  renormalized volume formula (\ref{RenVol}),  the proposition will follow from an application  of (\ref{L1 norm H-1}) and (\ref{asymptotics r}).	
\end{proof}

\subsection{Isoperimetric comparison}	 Recall that $\Omega_0$ denotes the outermost region of the convex co-compact hyperbolic manifold $M$. In what follows we will show an isoperimetric comparison for the outermost isoperimetric profile $I_{M}$
	 \[
	 I_{M}(V)= \inf \{ \text{area}(\partial \Omega)\,:\, \Omega_0\subset \Omega\quad \text{and}\quad \text{vol}(\Omega-\Omega_0)=V
	 \}
	 \] 
	with that of the hyperbolic metric with totally geodesic convex core. Notice that $I_{M}$ takes into consideration all ends of  $M$. The standard isoperimetric profile of $M$, i.e.  infimum of boundary area among  all regions of a given volume, is denoted by $J_M$. .

\begin{theorem}\label{main result}
Let M be a hyperbolic $3$-manifold that is either acylindrical or quasi Fuchsian. Let $I_{M}(V)$ be the isoperimetric profile of $M$ with respect to its outermost region $\Omega_0$ and $I_{TG}$ (resp. $J_{TG}$) the outermost  isoperimetric profile (resp. standard isoperimetric profile) of the hyperbolic metric in the deformation space of $M$ that has totally geodesic convex core $\Omega_{TG}$. Then  \begin{eqnarray*} \lim_{V\rightarrow \infty}\bigg(J_{TG}(V)-J_M(V)\bigg)>0 \quad \text{and}\quad I_{M}(V)< I_{TG}(V+|\Omega_0| - |\Omega_{TG}|).
\end{eqnarray*}
In particular, if $M$ contains only one closed minimal surface, then $I_M(V)<I_{TG}(V)$ for every volume $V>0$.
\end{theorem}
One should observe   by the special case in the main theorem in \cite{AST} that  $|\Omega_0|-|\Omega_{TG}|\geq 0$.
\begin{proof}
Let  $U_r$ be the equidistant foliation of $M$ at infinity  by convex sets inducing the hyperbolic   metric $h_0$ of Gauss curvature $-4$ at the conformal infinity $\partial M$  via
\[
h_0=\lim_{r\rightarrow \infty}\, e^{-2r} g_r,
\]	
where $g_r$ is the induced Riemannian metric on $\partial U_r$.  

Proposition \ref{prop:asympprofileVR} and an isoperimetric comparison yields 
\begin{eqnarray}\label{RenVolProfile}
 V_R(M)&=& \lim_{r\rightarrow \infty} \bigg(\text{vol}(U_r)- \frac{1}{2}\text{area}(\partial U_r)  + \pi\chi(\partial M)\log \sqrt{\frac{2\,Area(\partial U_r)}{\pi|\chi(\partial M)|}}\, \bigg) -\frac{\pi}{2}\chi(\partial M) \nonumber \\
 &\leq& \lim_{V\rightarrow \infty} \bigg( 
V- \frac{1}{2}J_{M}(V) +\pi\chi(\partial M)\log \sqrt{\frac{2\,J_M(V)}{\pi |\chi(\partial M)|}}
\bigg) -\frac{\pi}{2}\chi(\partial M)
\end{eqnarray}
where we are using that the function $x-\pi\chi(\partial M)\log\sqrt{x}$ is increasing for large values of $x$.

By Theorem \ref{Comparison Renormalized volume} and Proposition \ref{prop:asympprofileVR}, the Renormalized volume $V_R(M)$  as a functional in the moduli space of convex co-compact $3$-manifolds attains its  global minimum  at  the geodesic class of the conformal infinity $\partial M$.  Using  this  comparison  and 
\[
 V_R\big(g_{TG}\big)=  \lim_{V\rightarrow \infty} \bigg( 
V- \frac{1}{2}J_{TG}(V) + \pi \chi(\partial M)\log\, \sqrt{\frac{ 2\, J_{TG}(V)}{\pi \chi(\partial M)}}
\bigg) - \frac{\pi}{2}\chi(\partial M)
\]
we obtain that
\[
\lim_{V\rightarrow \infty} \bigg( (J_{TG}(V) -\pi\chi(\partial M)\log \sqrt{J_{TG}(V)}) - (J_M(V) - \pi\chi(\partial M)\log \sqrt{J_M(V)}) \bigg) \geq V_R(M)-V_R(M_{TG}).
\]
Since $x-\pi\chi(\partial M)\log\sqrt{x}$ is increasing for large $x$, we have that
\[
\lim_{V\rightarrow \infty} \bigg( (J_{TG}(V) - J_M(V) \bigg) > 0.
\]
Using  for large volume $V$  that $I_M(V) = J_M(V+|\Omega_0|)$ and  $I_{TG}(V) = J_{TG}(V+|\Omega_{TG}|)$, we conclude after a relabeling  that the isoperimetric profile of $M$ with respect to the outermost region $\Omega_0$ satisfies 
\[
\lim_{V\rightarrow \infty} \bigg( I_{TG}(V+|\Omega_0| - |\Omega_{TG}|)-I_{M}(V) \bigg)> 0.
\]
 As observed earlier, we have  as a  special case of the main theorem in \cite{AST} that  $|\Omega_0|-|\Omega_{TG}|\geq 0$.  In particular, we obtain \[I_M(0)\leq I_{TG}(0)<I_{TG}(0-|\Omega_0|- |\Omega_{TG}|).\] The first inequality follows from the Gauss-Bonnet Theorem and the second  from the monotonicity of the isoperimetric profile.
On the other hand, we have by definition of the Hawking mass function the following equation for $I_{M}$ and $I_{TG}$ 
	\[
-2\pi \mathcal{X}(M) - \frac{I_M^{\prime+}(V)^2I_M(V)}{4} +I_M(V)= \frac{m_H(V)}{\sqrt{I_M(V)}}.
	\]
	\[
2\pi \mathcal{X}(M) + \frac{I_{TG}^{\prime}(V+|\Omega_0| - |\Omega_{TG}|)^2I_{TG}(V+|\Omega_0| - |\Omega_{TG}|)}{4} -I_{TG}(V+|\Omega_0| - |\Omega_{TG}|)=0.
	\]
	By adding these two equations, we obtain
	\[
	I_M(V)-I_{TG}(V+ |\Omega_0| - |\Omega_{TG}|) - \frac{(I_M^{\prime+})^2I_M(V)-(I_{TG}^{\prime+})^2I_{TG}(V+|\Omega_0)| - |\Omega_{TG}|)}{4}= \frac{m_H(V)}{\sqrt{I_M(V)}}.
	\]
Therefore, the function $f(V)=I_M(V)-I_{TG}(V+|\Omega_0|-|\Omega_{TG}|)$ does not   have a positive local maximum point since that would imply  $m_H(V)>0$, contradicting   Lemma \ref{hawking mass negative}. While the isoperimetric profile $I_{TG}$ is  smooth,  the  graph of the isoperimetric profile $I_M$ can have corners.    To deal with this possibility at the local maximum point $V_0$,  we replace $I_M$ locally near $V_0$ by the area profile function $f(V)$ associated to the equidistant deformation of the isoperimetric surface $\Gamma_{V_0}$. 
\end{proof}

 Theorem \ref{main result} brings a connection between $V_R$ and $J_M(V)$. The following proposition    shows that $V_R$ is in fact determined by the asymptotic of the isoperimetric profile $J_M(V)$. 

\begin{theorem}[Theorem 1.2]\label{renormalized isoperimetric} Let $M$ be a convex co-compact hyperbolic 3-manifold. Then
\begin{eqnarray*}
\begin{split}
V_R(M) + \frac{\pi}{2}\chi(\partial M)&= \lim_{V\rightarrow \infty} \bigg( 
V-  \frac{1}{2}J_{M}(V) + \pi \chi(\partial M)\log \sqrt{\frac{2\,J_M(V)}{\pi |\chi(\partial M)|}}\,
\bigg).\\
&= \lim_{V\rightarrow \infty} \bigg( 
V + |\Omega_0|-  \frac{1}{2}I_{M}(V) + \pi\chi(\partial M)\log \sqrt{\frac{2\,I_M(V)}{\pi |\chi(\partial M)|}}\,
\bigg).
\end{split}
\end{eqnarray*}
\end{theorem}

\begin{proof}
Let $\Omega_V$ be the isoperimetric region of  volume $V$ and $\Omega_r$ its equidistant enlargement region at distance $r$. Following  Lemma \ref{cheeger}, we have   that
\begin{eqnarray}\label{eq1}
\lim_{r\rightarrow \infty} e^{-2r}|\partial \Omega_r|= \frac{1}{2}|\partial \Omega_V| + \frac{1}{2}\int_{\partial \Omega_V} H\, d\Sigma + \frac{\pi}{2} \mathcal{X}(\partial M)= \beta
\end{eqnarray}
By the computations in Lemma \ref{cheeger} and the variational characterization of the renormalized volume in the space of   metrics conformal to $\partial M$ having area $\beta$, we obtain
\begin{eqnarray*}
|\Omega_V|- \frac{1}{2}\int_{\partial \Omega_V} H\,d\Sigma + \frac{\pi}{2}\mathcal{X}(\partial M)=\lim_{r\rightarrow \infty} \bigg(|\Omega_r| - \frac{1}{2}|\partial \Omega_r| + r \pi \mathcal{X}(\partial M) \bigg)\leq V_R(M,\beta) + \frac{\pi}{2}\chi(\partial M).
\end{eqnarray*}
The notation $V_R(M,\beta)$ reflects the constraint on the area of the conformal metric at infinity. On the other hand, one has
 \[V_R(M,\beta)=  V_R\big(M, \frac{\pi}{2}|\chi(\partial M)|\big) - \frac{\pi}{2}\chi(\partial M)\log \bigg(\frac{2\beta}{\pi |\chi(\partial M)|}\bigg).\] 
 Therefore,
\begin{equation}\label{eq2}
|\Omega_V| - \frac{1}{2}|\partial \Omega_V|- \frac{1}{2}\int_{\partial \Omega_V} (H-1)\,d\Sigma + \frac{\pi}{2}\mathcal{X}(\partial M) + \frac{\pi}{2}\chi(\partial M) \log\bigg(\frac{2\beta}{\pi |\chi(\partial M)|}\bigg) \\ \leq V_R\big(M, \frac{\pi}{2}|\chi(\partial M)|\big) + \frac{\pi}{2}\chi(\partial M).
\end{equation}
Substituting equation (\ref{eq1}) into (\ref{eq2}), we obtain that

\begin{eqnarray}\label{take limit}
 V- \frac{1}{2}J_M(V) &+& \frac{\pi}{2} \mathcal{X}(\partial M) \log \bigg( \frac{2\,J_M(V)}{\pi |\chi(\partial M)|} + \frac{1}{|\pi\chi(\partial M)|}\int_{\partial \Omega_V} (H-1)\,da - 1\bigg) \nonumber \\
& \leq&  V_R\big(M,\frac{\pi}{2}|\chi(\partial M)|\big) + \frac{1}{2}\int_{\partial \Omega_V} (H-1)\,da - \frac{\pi}{2}\chi(\partial M) +\frac{\pi}{2} \chi(\partial M). 
\end{eqnarray} 
Since the Hawking mass $m_H(V)$ is monotone increasing by Lemma \ref{monotonicity} and bounded by Lemma \ref{hawking mass negative}, we  have
\[
\lim_{V\rightarrow \infty} \int_{\Sigma_V} \big(H-1 \big) d\Sigma=  \pi \mathcal{X}(\partial M).
\]
By Taking the limit as $V\rightarrow \infty$ in both sides of (\ref{take limit}) and applying this identity, we  obtain 
\begin{eqnarray*}
	 \lim_{V\rightarrow \infty} \bigg( 
	V-  \frac{1}{2}J_{M}(V) + \pi\chi(\partial M)\log \sqrt{\frac{2\,J_M(V)}{\pi |\chi(\partial M)|}}\,
	\bigg) \leq V_R(M) + \frac{\pi}{2} \chi(\partial M).
\end{eqnarray*}
The reverse inequality, obtained in (\ref{RenVolProfile}), is a straightforward  isoperimetric comparison for the terms in the expression of $V_R(M,h_0)$.
\end{proof}

\begin{remark}
The  following isoperimetric inequality  for strongly convex regions $\Omega\subset M$ such that $\partial \Omega \in  [\partial M]$  follows from  the definition of the renormalized isoperimetric constant $V_R(M)$ and Proposition \ref{renormalized isoperimetric}:
\begin{equation*}\label{renormalized isoperimetric constant}
|\Omega| - \frac{1}{2} \int_{\partial \Omega} H\, da \leq V_R(M) + \frac{\pi |\chi(\partial M)|}{2} \log \bigg( \frac{1}{\pi |\chi(\partial M)|}\int_{\partial \Omega} H\,da  + \frac{|\partial \Omega|-\pi |\chi(\partial M)|}{\pi |\chi(\partial M)|} \bigg).
\end{equation*}
\end{remark}

\begin{corollary}[Theorem 1.1]\label{cor:maintheorem1}
	Let $M$ be a convex co-compact $3$-manifold that is either acylindrical or quasifuchsian, and let $\Omega_0$ be its outermost region. If $I_{M}, I_{TG}$ denote the outermost isoperimetric profiles of $M$ and $M_{TG}$ (the quasiconformal deformation of $M$ with Fuchsian ends) respectively, then
	\begin{eqnarray*}
		\frac12\lim_{V\rightarrow \infty} \big(I_{TG}(V)- I_{M}(V)\big)= V_R(M)-|\Omega_0|.
	\end{eqnarray*}
	Equivalently,\begin{eqnarray*}
V_R(M)=	\frac12\lim_{V\rightarrow \infty} \big(J_{TG}(V)- J_{M}(V)\big).
\end{eqnarray*}
 \end{corollary}

\begin{proof}
Using Theorem \ref{renormalized isoperimetric} for $M$ with $V=V'+|\Omega_0|$ and for $M_{TG}$ with $V=V'+|\Omega_{TG}|$ and then relabeling $V'$ by $V$ we obtain

\begin{equation}\label{eq:diffVR}
    V_R(M) - V_R(M_{TG}) = \lim_{V\rightarrow\infty} \bigg( |\Omega_0|-|\Omega_{TG}| + \frac12(I_{TG}(V)-I_M(V)) + \pi\chi(\partial M)\log\sqrt{\frac{I_{TG}(V)}{I_M(V)}} \bigg)
\end{equation}
Since $V_R(M_{TG}) = |\Omega_{TG}|$ then this reduces to 
\[
V_R(M) = |\Omega_0| +  \lim_{V\rightarrow\infty} \bigg( \frac12(I_{TG}(V)-I_M(V)) + \pi\chi(\partial M)(\log\sqrt{I_{TG}(V)}- \log\sqrt{I_M(V)}) \bigg)
\]

This in particular implies that $\limsup_{V\rightarrow\infty} |I_{TG}(V)- I_M(V)| < +\infty$, since otherwise the left-side of (\ref{eq:diffVR}) would not converge. We also have $\lim_{V\rightarrow\infty} I_{TG}(V) = \lim_{V\rightarrow\infty}I_M(V) = +\infty$, then $\lim_{V\rightarrow\infty} \log(\sqrt{I_{TG}(V)})-\log(\sqrt{I_M(V)}) = 0$. Hence we have

\begin{equation}
    V_R(M) = |\Omega_0|+\lim_{V\rightarrow\infty}  \frac12(I_{TG}(V)-I_M(V))
\end{equation}
which finishes the proof.
\end{proof}

\begin{remark} Observe that in Corollary \ref{cor:maintheorem1} $M_{TG}$ is not uniquely defined if $M$ is quasifuchsian, but since $I_{TG}$ is independent from the Fuchsian model considered, we proceed as normal. This remark remains valid for later statements.
\end{remark}

\begin{corollary}[Theorem 1.3]\label{cor:comparisonprofiles}
	Let $M$ be a convex co-compact hyperbolic $3$-manifold that is either acylindrical or quasifuchsian and $\Omega_0$ its outermost region.  If	$V_R(M)>|\Omega_0|$, then $I_M(V)<I_{TG}(V)$ for every volume $V\geq 0$. 
\end{corollary}
\begin{proof}
If $V_R(M)> |\Omega_0|$, then $\lim_{V\rightarrow \infty}I_M(V)-I_{TG}(V)>0$. We also have $I_{TG}(0)-I_M(0)>0$ by an application of the Gauss-Bonnet Theorem. The proof of  Theorem \ref{main result} applies verbatim when $I_{TG}(V+|\Omega_0|)$ is  replaced by $I_{TG}(V)$ to show that $I_{TG}-I_M$ cannot have a non-negative local maximum. Therefore, $I_M(V)< I_{TG}(V)$ for every $V$.
\end{proof}

\begin{remark}
	In order to see that Corollary \ref{cor:comparisonprofiles} is not an empty statement, observe that Theorem \ref{Comparison Renormalized volume} (\cite{BBB},\cite{P}) implies $V_R(M)>0$ when $M$ is  quasi Fuchsian (but not Fuchsian) and contains a unique compact minimal surface. This last condition (unique minimal surface) is non-empty, since it contains in its interior the set of \textit{almost-Fuchsian} manifolds (studied by Uhlenbeck \cite{U}) , which are manifolds that contain a minimal surface with principal curvatures $|k_{1,2}|<1$.
\end{remark}

\begin{question}It is natural  to ask what is the reach of Corollary \ref{cor:comparisonprofiles}, particularly since the comparison between $V_R(M)$ and $\Omega_0$ answers if $I_M$ ever surpasses $I_{TG}$ or not.
For which $M$ convex co-compact hyperbolic $3$-manifold we have that $V_R(M)>|\Omega_0|$?
\end{question}

\begin{remark} Results up until this Section hold for relatively acylindrical hyperbolic manifolds $(M^3,S\subseteq\partial M)$ when one considers the definitions of isoperimetric profile and renormalized volume for the set of ends $S\subseteq\partial M$ (for example one can consider the case of a Bers slice). Similarly, one can write the results for convex co-compact manifolds with incompressible boundary (but not necessarily acylindrical of quasifuchsian), with the caviat that the isoperimetric model $I_{TG}$ is replaced by a disjoint union of Fuchsian ends. If the boundary is not incompressible, the existence of an outermost minimal core needs to be assumed for the results to work. We focused in the acylindrical/quasifuchsian case to illustrate the isoperimetric features of renormalized volume while avoiding the technical setup required to establish a more general statement.
\end{remark}

\section{Minkowski inequality for Horospherically convex sets}\label{sec:minkowskiinequality}
\noindent In this section we are interested in studying the geometric objects in $\mathbb{H}^3$ that, according to Epstein's description, correspond to  conformal metrics in $\partial_{\infty}\mathbb{H}^3= \mathbb{S}^2$.
\begin{definition}[\cite{EGM}]
A   hypersurface  $\Sigma$ in 
 $\mathbb{H}^3$ is said to be  \textit{horospherically convex} ($h$-convex) if at every point $\Sigma$ lies locally on one side of its   tangential horosphere. 
\end{definition}

\noindent An oriented surface $\Sigma\subset \mathbb{H}^3$ is horospherically convex at $p\in \Sigma$
if, and only if, all the principal curvatures of $\Sigma$ at $p$ verify simultaneously $\kappa_i(p) < -1$ or
$\kappa_i(p) > -1$, see \cite{EGM}.
Here we assuming that the orientation of $\Sigma$ coincide with the outward orientation of the horosphere tangent to  $\Sigma$ at $p$.  This definition is more general than geodesic convexity. If $\Sigma$ is closed and lies in the concave side of the tangential horosphere at each point for example, then the  principal curvatures  satisfy $\kappa_i>-1$.
\vspace{0.1cm}

\noindent If $\Sigma$ is horospherically convex bounding a compact region $\Omega$, then the  outward exponential map $\psi: \Sigma\times [0,\infty) \rightarrow \mathbb{H}^3- \Omega$, given by $\psi(p,r)=exp(p,rN)$, is a diffeomorphism.  The  family $\Sigma_r=\psi(\Sigma,r)$ are called  the normal flow of  $\Sigma$ . The  foliation $\{\Sigma_r\}$ induces a  metric $h$ at the conformal boundary $\partial_{\infty} \mathbb{H}^3=\mathbb{S}^2$  by
\[
h= \lim_{r\rightarrow \infty} e^{-2r} g_r
\]
where $g_r$ is the first fundamental form of the parallel surface $\Sigma_r$. 
More importantly, for each metric $h$ in the conformal class at infinity  there exist an unique equidistant foliation  such that the associated metric at infinity is $h$. This well known result has its root in the work of  C. Epstein \cite{E} through the envelopes of horospheres construction which we briefly describe:
\vspace{0.1cm}

\noindent Consider the Poincare ball model for $\mathbb{H}^3$. For any $x\in \mathbb{H}^3$ define its \textit{visual metric} $v_x$ as a metric in the conformal $\partial_\infty \mathbb{H}^3 = S^2$ by
\begin{enumerate}
    \item $v_0$ is the canonical round metric in $S^2$.
    \item If $\gamma$ is an isometry of $\mathbb{H}^3$ so that $\gamma(x)=0$, then $v_x=\gamma^*(v_0)$ 
\end{enumerate}
Observe that $v_x$ is well defined because isometries of $\mathbb{H}^3$ fixing the origin $0$ are the isometries of the round metric in $S^2$.
 It is not a hard exercise to see that for $x\in\mathbb{H}^3, b\in S^2$ the set $H(b,v_x(b)) :=\lbrace y\in\mathbb{H}^3\,|\,v_y(b)=v_x(b) \rbrace$ is the horosphere tangent at $b$ passing through $x$.
 \vspace{0.1cm}

\noindent Epstein \cite{E} shows that given a $C^1$ conformal metric $\rho$ in a open set $U\subseteq S^2$, there exists a unique continuous map $Y_\rho:U\rightarrow T^1\mathbb{H}^3$ so that $Y(b)$ is a unit normal vector to $H(b,\rho(b))$ oriented towards $b$. This map satisfies that for $t$ constant, $Y_{e^t\rho}$ is equal to $Y_\rho$ after translating $t$ units by the geodesic flow. Moreover, if $\rho$ is smooth and we fix some compact $K\subset U$, we have that $Y_{e^t}|_K$ is an embedding for sufficiently large $t$. Hence if we take a convex co-compact manifold $M$ and a metric $h$ in the conformal boundary $\partial_\infty M$, the maps $Y_{e^th}(\partial M)$ are well defined, and define a equidistant foliation of the ends of $M$ when $t$ is sufficiently large.
\vspace{0.1cm}

\noindent In the rest of this section we will   see that the $W$-volume is maximized among metrics of  fixed area by constant curvature metrics. This result can also be established  from the work of Osgood, Phillips and Sarnak \cite{OPS}, where they show that constant curvature metrics in $S^2$ maximize $\log(\text{Det}(\Delta))$, the logarithm of the determinant of the Laplace-Beltrami operator. The result for $W$-volume follows because its first variation formula is a constant multiple ($3\pi$ in fact) of the first variation formula of $\log(\text{Det}(\Delta))$. Our proof is based on the renormalized Ricci flow for surfaces.

\subsection{Minkowski-type inequality}

The following Minkowski-type inequality is obtained by comparing the quantity $|\Omega|-\frac12\int_{\partial\Omega} H\,da$ for a given compact region with horospherically convex boundary with that of a geodesic ball via the Renormalized Ricci flow for conformal metrics in $\partial_\infty\mathbb{H}^3=S^2$ (as done in \cite[Section 3, Theorem 2.A]{OPS}, \cite[Section 4]{VP} for $\chi(\Sigma)\leq0$).

\begin{theorem}\label{isoperimetric inequality}
If $\Sigma$ is an horospherically convex surface bounding a compact region $\Omega \subset \mathbb{H}^3$, then
\begin{eqnarray*}
	\int_{\Sigma} H\,d\Sigma - 2|\Omega| \, \geq\, 2\pi \log\bigg(1 + \frac{1}{2\pi}\int_{\Sigma}(H+1)d\Sigma \bigg) 
\end{eqnarray*}
with equality if, and only if, $\Sigma$ is a geodesic sphere.
\end{theorem}
The result without the rigidity statement was also obtained by J. Nat\'{a}rio \cite{N}.

\begin{proof}
Given a horospherically convex domain $\Omega\subset \mathbb{H}^3$, we consider the $W$-volume functional 
\[
W(\Omega)= |\Omega|-\frac12\int_{\Sigma} H\,d\Sigma.
\]
By mean of the correspondence between  equidistant foliation and metrics at infinity, one can prove that the function $W(\Omega)$ depends only on the metric $h\in [\partial \mathbb{H}^3]$. The first variation for $W$ for conformal deformations of $h$ was computed in \cite{KS}:
\[
\delta W(\hat{h})= \frac14 \int_{\mathbb{S}^2} \delta K(\hat{h})\,dvol_h,
\]
where $K$ is the Gauss curvature of $(\mathbb{S}^2,h)$. In other words, the Ricci flow is the gradient-like flow for the functional $W$. In particular, if $\partial_t h_t= (\frac{8\pi}{area(h)} - 2K)h_t$ is the Renormalized Ricci Flow that keeps the area of $h$ constant, then
\[
\delta W(h)= \frac14 \int_{\mathbb{S}^2} \bigg(\Delta K + K\left(2K - \frac{8\pi}{area(h)} \right) \bigg) dvol_h=  \frac12\int_{\mathbb{S}^2} K^2\, dvol_h -   \frac{1}{2area(h)} \bigg(\int_{\mathbb{S}^2} K\,dvol_h\bigg)^2.
\]
It follows from the Cauchy-Schwarz inequality that $\delta W(h)\geq 0$. By the strong convergence results for the Renormalized Ricci flow \cite{H}, we conclude that the $W$-functional has a global maximum among conformal metrics of a fixed area in the conformal boundary $[\partial M]$   at constant curvature metrics. These round metrics correspond to the normal flow of geodesic spheres in $\mathbb{H}^3$.  

If $h$ is the conformal  metric at infinity for  the equidistant foliation associated to  $\Omega$, then Lemma \ref{cheeger} implies that 
\[
area(h)= \frac{1}{2}|\Sigma| + \frac{1}{2} \int_{\Sigma}H\,d\Sigma + \pi.
\]
Therefore, if $B_r$ is a geodesic ball such that $|\partial B_r|+ \int_{\partial B_r} H_r= |\Sigma| + \int_{\Sigma}H\,d\Sigma=4\pi\lambda$, then 
\[
\int_{\Sigma} H\, d\Sigma -2|\Omega| \geq  \int_{\partial B_r} H_r\, dS_r - 2|B_r|.
\]
Note that $H_r= \frac{\cosh(r)}{\sinh(r)}$, $|\partial B_r|= 4\pi \sinh^2(r)$, and $|B_r|= \pi\sinh(2r) - 2\pi r$. Using these formulas we obtain 
\[
\int_{\partial B_r} H_r\, dS_r - 2|B_r|= 4\pi r
\]
and  $\lambda= \cosh(r)\sinh(r) + \sinh^2(r)$. From this  last equality we deduce that $r= \sinh^{-1}\big(\frac{\lambda}{\sqrt{1+2\lambda}}\big)$. Since $\sinh^{-1}(x)= \log \big(\sqrt{1+x^2} + x \big)$, we conclude that $r=\frac{1}{2}\log(1+2\lambda)$. Therefore,
\[
\int_{\Sigma} H\, d\Sigma -2|\Omega| \geq  2\pi \log\bigg(1+ \frac{1}{2\pi}\int_{\Sigma} \big(H+1\big)\,d\Sigma \bigg),
\]
with equality if, and only if, $\Sigma$ is up to a rigid motion the geodesic sphere $\partial B_r$.
\end{proof}

\begin{remark}  For an horospherically convex surface surface $\Sigma$ bounding a compact region $\Omega\subset \mathbb{H}^3$, there is  another sharp Minkowski inequality  \cite{GWW}:
 \begin{equation}\label{Minkowski}
 \int_{\Sigma} H\,d\Sigma\, \geq \,  2\pi \sqrt{\frac{|\Sigma|^2}{4\pi^2}  + \frac{|\Sigma|}{\pi}}
 \end{equation}
 Combining  (\ref{Minkowski}) with the inequality in Theorem \ref{isoperimetric inequality}, we obtain 
\begin{equation}\label{Minkowski II}
	\int_{\Sigma} H\,d\Sigma \,\geq \,  2|\Omega|+  2\pi \log\bigg(1 + \frac{|\Sigma|}{2\pi} + \sqrt{\frac{|\Sigma|^2}{4\pi^2}  + \frac{|\Sigma|}{\pi}} \,\bigg),
\end{equation}
 with equality if, and only if, $\Sigma$ is a geodesic sphere.
  \vspace{0.2cm}
 
 \noindent 
 In contrast with Theorem  \ref{isoperimetric inequality}, whose poof relies on $2$-dimensional features of the renormalized volume, the proof of the Minkowski inequality (\ref{Minkowski}) involves a mean curvature  type flow  and  generalizes to higher dimensions \cite[Theorem 6.1]{GWW}.
\end{remark}

\subsection{Polyakov-type formula}

As stated in \cite{GMS} we have the following Polyakov type formula for conformal metrics in the sphere

\begin{equation}\label{eq:polyakov}
    W(e^{2\omega}h_0) - W(h_0) = -\frac14 \int_{S^2} |\nabla\omega|^2_{h_0} + \text{Scal}_{h_0}\omega \text{dvol}_{h_0}.
\end{equation}
This formula follows, as in the geometrically finite case, by applying the Fundamental Theorem of Calculus and the first variation formula for the $W$-volume to the 1-parameter family of metrics $h_t=e^{2t\omega}h_0$, see Proposition \ref{polyakov} in the Appendix.

As done in \cite[Proposition 3.11]{Schlenker} for the convex co-compact case  with $\chi(\partial M)<0$, we have the following monotonicity for the $W$-volume. In this case $W$-volume is monotone decreasing, contrary to \cite[Proposition 3.11]{Schlenker}, which boils down to the signature of the Euler characteristic of the boundary.

\begin{proposition}\label{prop:Wmono}
Let $h_0, h_1$ be non-negatively curved conformal metrics on $\partial_{\infty}\mathbb{H}^3=S^2$ so that $h_0\leq h_1$ pointwise. Then

\[W(h_0) \geq W(h_1).
\]
Moreover, equality occurs if and only if $h_0=h_1$ pointwise.
\end{proposition}

\begin{proof}
Define $\omega:S^2\rightarrow\mathbb{R}$ so that $h_1=e^{2\omega}h_0$. Since $h_0\leq h_1$, then $\omega\geq0$. Hence, we have $\text{Scal}_{h_0}\omega\geq0$ pointwise. By the Polyakov-type formula (\ref{eq:polyakov})

\begin{equation}
    W(h_1) - W(h_0) = -\frac14 \int_{S^2} |\nabla\omega|^2_{h_0} + \text{Scal}_{h_0}\omega\, \text{dvol}_{h_0} \leq 0,
\end{equation}
so the inequality follows.

For the equality, note that $\int_{S^2} |\nabla\omega|^2_{h_0}\text{dvol}_{h_0}= \int_{S^2}\text{Scal}_{h_0}\omega \text{dvol}_{h_0}=0$. Hence $\omega$ is a constant function, which  by the Gauss-Bonnet yields $\int_{S^2}\text{Scal}_{h_0}\omega \text{dvol}_{h_0} = 8\pi\omega$. Therefore, $\omega=0$ and $h_0=h_1$ pointwise.
\end{proof}

Proposition \ref{prop:Wmono} can be written in terms of horospherically convex spheres in $\mathbb{H}^3$.

\begin{corollary}\label{monotonicity horospherically convex}
Let $\Sigma_0, \Sigma_1$ be horospherically convex spheres in $\mathbb{H}^3$ bounding regions $\Omega_0, \Omega_1$ so that $\Omega_0 \subset \Omega_1$. If  $\text{Scal}_{\Sigma_0}\geq 0$, then

\[ \int_{\Sigma_0}Hd\Sigma_0 - 2|\Omega_0| \leq \int_{\Sigma_1}Hd\Sigma_1 - 2|\Omega_1|,
\]
where equality occurs if and only if $\Sigma_0, \Sigma_1$ are the same surface.
\end{corollary}
\begin{proof}
Let $h_0, h_1$ be the conformal metrics in $\partial_{\infty}\mathbb{H}^3=S^2$ corresponding to $\Sigma_0, \Sigma_1$, respectively. Since $\Sigma_0\subset\Omega_1$ then $h_0\leq h_1$, because any outer-tangent horosphere to $\Sigma_1$ will not intersect $\Sigma_0$. 

If $k_{1,2}(p)$ are the principal curvatures of $\Sigma_0$, then the scalar curvature at $p^+$ (point at infinity whose outer-tangent horosphere to $\Sigma_0$ is tangent at $p$) is given by $\frac{-1+k_1(p)k_2(p)}{(1+k_1(p))(1+k_2(p))}$. Hence $\text{Scal}_{h_0}(p^+)\geq 0$ if and only if $(-1+k_1(p)k_2(p))\geq 0$, which by Gauss equation is equivalent to $\text{Scal}_{\Sigma_0}(p)\geq0$.

Hence we have met the conditions of Proposition $\ref{prop:Wmono}$, from which the result follows.
\end{proof}

\noindent Finally, let's observe how the Polyakov  formula (\ref{eq:polyakov}) relates to Theorem \ref{isoperimetric inequality}. Assume  that $h_0$ is a conformal metric in $\partial_{\infty}\mathbb{H}^3=S^2$ with constant scalar curvature $K>0$ and take $\omega:S^2\rightarrow\mathbb{R}$ so that $\int_{S^2} (e^{2\omega}-1) \text{dvol}_{h_0}=0$.  In other words,  the conformal metric $h_1=e^{2\omega}h_0$ and  the constant curvature metric $h_0$ have the same area. As detailed in \cite[Lemma7]{CM}, this assumption implies that $\int_{S^2} \omega d\text{vol}_{h_0}\leq 0$. The proof of Theorem \ref{isoperimetric inequality} show  for $h_1=e^{2\omega}h_0$ that $W(h_1)\leq W(h_0)$ with equality if, and only if, $h_1$ has constant scalar curvature $K$.  Note that unlike  the convex co-compact case, this fact does not follow from (\ref{eq:polyakov}). Consequently, the Polyakov formula (\ref{eq:polyakov}) yields:

\begin{corollary}
	Let  $h_0$ be the round metric in $\partial_{\infty}\mathbb{H}^3=\mathbb{S}^2$ with constant Gauss curvature $1$, and take $\omega:S^2\rightarrow\mathbb{R}$ so that $\int_{S^2} (e^{2\omega}-1) \text{dvol}_{h_0}=0$. Then
	\[ \Bigg| \int_{\mathbb{S}^2} 2\omega\, \textrm{dvol}_{h_0} \Bigg| \leq \int_{\mathbb{S}^2} |\nabla\omega|^2_{h_0}\, \text{dvol}_{h_0}
	\]
	Moreover, equality occurs if and only if $h_1=e^{2\omega}h_0$ has constant Gauss curvature $1$, or equivalently, $e^{2\omega}$ is given by the derivative of a M\"{o}bius transformation.
\end{corollary}
\begin{remark}
	A complete proof of this result with the line of reasoning mentioned above is done in \cite[Section 2.3]{OPS} using the $\log(\text{Det}(\Delta))$ functional. As mentioned earlier,  the first variation formula of $\log(\text{Det}(\Delta))$ is a constant multiple  of the first variation formula of the $W$-volume.
\end{remark}

\section{Appendix}

\subsection{Polyakov formula for the $W$-volume} In this section we prove the Polyakov formula from the first variation formula for the $W$-volume as a functional on the conformal boundary class   at infinity.
\begin{proposition}\label{Krasnov-Schlenker}
The  $W$-volume of a  convex set $\Omega$ inside a convex co-compact hyperbolic $3$-manifold $M$  depends only on the metric at   infinity $h$ associated to $\Omega$. Moreover, the first variation  among  conformal deformations  is \[
	\delta W(\hat{h})= \frac{1}{4} \int_{\partial M} \delta K(\hat{h})\,dh.
	\]
\end{proposition}
\begin{proof}
	See Krasnov-Schlenker \cite[Section 7]{KS}.
\end{proof}
\begin{proposition}\label{polyakov}
The $W$-volume satisfies the Polyakov-type formula:
\[
W(e^{2\omega}h_0) - W(h_0)= - \frac{1}{4} \int_{\partial M} |\nabla\omega|^2_{h_0} + \text{Scal}_{h_0}\omega\, \text{dvol}_{h_0}.
	\]
\end{proposition}
\begin{proof}
	Let $h_t=e^{2t\omega}h_0$ and $t\in[0,1]$. Applying  the Fundamental Theorem of Calculus and Proposition \ref{Krasnov-Schlenker}, we obtain
	\begin{eqnarray*}
W(e^{2\omega}h_0) - W(h_0)&=& \int_{0}^{1} \delta W(h_t)(2\omega \,h_t)\, dt=   \frac{1}{4}\int_{0}^{1}	\int_{\partial M} \delta K_{h_t}(2\omega \,h_t) dh_t\, dt \\
&=& \frac{1}{4} \int_{0}^{1} \int_{\partial M} \big(-2\omega e^{-2t\omega} K +2\omega e^{-2t\omega}\Delta (t\omega)- e^{-2t\omega}\Delta \omega  \big) dh_t\, dt \\
&=& \frac{1}{4} \int_{0}^{1} \int_{\partial M} \big( -2\omega\,K + 2t\,\omega \Delta \omega \big) dh_0\, dt  \\
&=& \frac{1}{4} \int_{\partial M} \big( -2\omega\, K- |\nabla \omega|^2\big)\,dh_0.
	\end{eqnarray*} 
\end{proof}
\subsection{Free boundary stability  for cmc surfaces between two parallel geodesic planes}
 In this section we study the isoperimetric problem for  regions bounded by two  geodesic planes in $\mathbb{H}^3$. In contrast with previous works for slabs,  the Alexandrov reflection principle is not available in this setting. Instead,  we   exploit the reflections across the   geodesic boundary and a result of Hsiang \cite{H} to  reduce the problem to rotationally invariant surfaces. With this simplification, we are able to  extend a characterization result of Ritor\'{e}-Ros \cite{RR} to the free boundary case discussed here. As the classical result for slabs in $\mathbb{R}^3$, we will  show  that  geodesic spheres and tubes about geodesics are the only  solutions.  The  isoperimetric problem in cyclic quotients of $\mathbb{H}^3$ is also treated.

\noindent We start describing  the basic cyclic  actions in $\mathbb{H}^3$ by isometries  and the main model for the slabs between geodesic planes. Let $G_a$ and $G_{\lambda}$ be cyclic subgroups of  $Iso(\mathbb{H}^3)$    generated respectively by the isometries  \[\gamma_1(x,y,z)=(x+a_1,y+a_2,z) \quad \text{and} \quad \gamma_2(x,y,z)=(\lambda (x-x_0), \lambda (y-y_0), \lambda z).\]

\noindent Let $M$ denote either  the slab bounded by two vertical planes with Euclidean distance  $a$ or the slab bounded by two concentric hemispheres perpendicular to $\partial \mathbb{R}_{+}^3$ with Euclidean distance  $\lambda$. Without loss of generality, we can assume that $M$ is  the fundamental domain of the cyclic group $G_a$ or $G_{\lambda}$, respectively.

\begin{theorem}\label{stable cmc between geodesic planes}
Let $\Sigma$ be a free boundary stable constant mean curvature surface in $M$. If the distance between the geodesic  planes of $\partial M$ is positive,  then $\Sigma$ is either a geodesic hemisphere or a tube about a free boundary geodesic connecting the two planes. If the distance is zero, then $\Sigma$ is a geodesic hemisphere.
\end{theorem}
\begin{proof}
 If  $\partial \Sigma$ has only one component, then $\Sigma$ is a geodesic hemisphere by Alexandrov's Theorem. Since $\partial M$ is totally geodesic, we can apply a hyperbolic reflection across one boundary component to obtain another free boundary  surface between hyperbolic planes. This compact extended surface can now be iterated infinity many times,  using the isometry $\gamma_i\circ \gamma_i$, $i=1,2$, to obtain a  complete  cmc surface $\hat{\Sigma}$ properly embedded in $\mathbb{H}^3$ that is cylindrically bounded.    By  Hsiang \cite{H}, $\hat{\Sigma}$ is a rotationally invariant surface and so does $\Sigma$. The result   will  follow from Ritor\'{e}-Ros \cite{RR}  as  sketched below:
 
The Hopf holomorphic quadractic differential applied to $(\Sigma,ds^2)$ implies that the metric $ds_0^2= b|A^{2,0}|$, $b^2=4(-1+H^2)$, is flat and conformal do the metric $ds^2$. Note by the maximum principle comparison with a flat tube  that $H>1$. If $G$ is the cyclic subgroup generated by $\gamma_i\circ \gamma_i$, then $\hat{\Sigma}/G$ has genus one in $\mathbb{H}^3/G$ and, hence, is without umbilical points. In particular, $ds_0^2$ is a smooth metric.  If we write $ds^2=\frac{e^{2\omega}}{b^2}\, ds_0^2$, then it is well known that $\omega$ satisfies the Sinh-Gordon equation \[\Delta_0 \omega + \sinh(\omega)\cosh(\omega)=0\]
It is noted in \cite{RR} that $\omega$ and the Gauss curvature $K$ of $\Sigma$ share the same sign. Arguing by contradiction, let us assume that  $\Omega_{1}$ and $\Omega_{2}$  are the sign components of $\omega$. Ritor\'{e}-Ros \cite{RR} considered the test function $f=a_1\sinh(\omega)$  in $\Omega_{1}$ and $f= a_2\sinh(\omega)$ in $\Omega_{2}$; $a_1$ and $a_2$ are chosen so that $f$ has mean zero on $\Sigma$. Hence, 
\begin{equation}\label{stability free boundary}
0\leq I(f,f)= - \int_{\Sigma} f \Delta f + (-2 + |A|^2)f^2 dA + \int_{\partial \Sigma} f \frac{\partial f}{\partial \nu}\, d\sigma  - \int_{\partial \Sigma} \Pi_{\partial M}(N,N) f^2\, d\sigma
\end{equation}
The last integral in (\ref{stability free boundary})  is zero since $\partial M$ is totally geodesic. Since $\Sigma$ can be reflected across  $\partial M$ and $\omega$ depends only at the geometric data of $\Sigma$, the second integral is also zero. For the  first integral, we follow computation in \cite{RR} to obtain
\begin{eqnarray*}
	0\leq I(f,f)&=& - \int_{\Sigma} f \Delta_0 f + (\cosh^2(\omega)+ \sinh^2(\omega))f^2 dA_0  \\
	&=& - \sum_{i=1}^{2} \int_{\Omega_i} a_i^2\,\sinh^2(\omega)\big(\sinh^2(\omega)+ |\nabla_0\omega|^2 \big)\, dA_0 <0.
\end{eqnarray*}
Hence, the Gauss curvature of $\Sigma$ has a sign and by the Gauss-Bonnet Theorem we conclude that $\Sigma$ is flat.
\end{proof}
\begin{corollary}	If $\Sigma$ is a compact embedded stable cmc surface in  $\mathbb{H}^3/G_{\lambda}$, then $\Sigma$ is either a geodesic sphere or a tube about a  closed geodesic. If $\Sigma$ is a compact embedded stable cmc surface in $\mathbb{H}^3/G_a$, then $\Sigma$ is a geodesic sphere.
\end{corollary}

\begin{remark}
	The isoperimetric problem between   horospheres is studied in \cite{CDP}. Building on the  symmetries of this setting and the Alexandrov reflection method,  the authors   show that  isoperimetric surfaces are rotationally invariant and classified those  that   meet the boundary orthogonally. The stability of these surfaces is not investigated and the possibility of onduloids  be isoperimetric for certain volumes is left open.
\end{remark}

\subsection{Stability of tubes} Let $\phi(r,\theta)=(e^r\cos(\theta), e^r\sin(\theta), a\,e^r)$ be the parametrization of a  tube $T$ of radius  $R= \log \big(\frac{1+ \sqrt{a^2+1}}{a}\big)$  about the  $z$-axis in $\mathbb{H}^3$. The metric and the second fundamental form of $T$ are
\begin{equation}\label{mean curvature tubes}
g = 
\begin{pmatrix}
\frac{1+a^2}{a^2} & 0  \\
0 & \frac{1}{a^2} 
\end{pmatrix}
\quad \text{and}\quad A = 
\begin{pmatrix}
\sqrt{a^2+1} & 0  \\
0 & \frac{1}{\sqrt{a^2+1}}
\end{pmatrix}
\end{equation}
Note that the mean curvature of $T$ satisfies  $H>1$. The Jacobi operator in this coordinate system is \[L= \frac{a^2}{1+a^2}\partial_{rr}+ a^2 \partial_{\theta\theta} + \big(a^2 + \frac{1}{a^2+1}-1\big).\]

 Let's look at the free boundary stability of $T$ between the  hyperbolic planes $z= \sqrt{1-x^2+ y^2}$ and $z=\sqrt{e^{2\lambda}-x^2-y^2}$. For this we look at the eigenvalue problem: 
\[
\bigg(\frac{a^2}{1+a^2} \partial_{rr}+ a^2\partial_{\theta\theta}\bigg) \varphi+ \big(a^2 + \frac{1}{a^2+1}-1\big)\varphi= - \lambda \varphi \quad \text{and}\quad \frac{\partial\varphi}{\partial r}(0,\theta)=\frac{\partial\varphi}{\partial r}(\lambda,\theta)=0
\] 
The constant functions are the first eigenfunctions with  eigenvalue  $\lambda_1= 1- a^2- \frac{a^2}{a^2+1}<0$. The stability  of $T$ is then equivalent to have the second eigenvalue $\lambda_2\geq 0$. Eigenfunctions of the form $\varphi(r,\theta)=\varphi(\theta)$ contribute positive eigenvalues for the Jacobi operator. An analysis of  the eigenfunctions $\varphi(r,\theta)=\varphi(r)$ shows that $\lambda_2\geq 0$ if, and only if,  $a\leq \frac{\pi}{\lambda}$. Note that the  distance between those two planes is $d_{\mathbb{H}}=\lambda$. Therefore,  the tube $T$ is stable if, and only if,
\begin{equation}\label{radius of stability}
R\geq \log \bigg(\frac{d_{\mathbb{H}}}{\pi} + \sqrt{\frac{d_{\mathbb{H}}^2}{\pi^2}+ 1}\bigg).
\end{equation}
 Let's look now at the volume preserving stability of $T$ in the cyclic quotient $\mathbb{H}^3/G_{e^{\lambda}}$.  For this we look at the eigenvalue problem: 
\[
\bigg(\frac{a^2}{1+a^2} \partial_{rr}+ a^2\partial_{\theta\theta}\bigg) \varphi+ \big(a^2 + \frac{1}{a^2+1}-1\big)\varphi= - \lambda \varphi \quad \text{and}\quad \varphi(0,\theta)=\varphi(k\lambda, \theta)\quad \forall k\in\mathbb{Z}
\] 
Following the discussion above, we have that  $\lambda_1= 1- a^2- \frac{a^2}{a^2+1}<0$ corresponding to the constant functions. The stability  of $T$ in $\mathbb{H}^3/G_{e^{\lambda}}$ is then equivalent to     $\lambda_2\geq 0$. As before, it is enough to consider eigenfunctions of the form $\varphi(r,\theta)=\varphi(r)$ which under the constraint above implies that 
$
a \leq \frac{2\pi}{\lambda}.
$

\subsection{Isoperimetric regions between two parallel geodesic planes} Up to a hyperbolic reflection, any region  bounded by two parallel geodesic planes of positive distance apart is congruent to the slab $M$ bounded by the hyperbolic planes $z= \sqrt{1-x^2+ y^2}$ and $z=\sqrt{e^{2\lambda}-x^2-y^2}$ for some $\lambda$. 
\vspace{0.2cm}

\noindent Let us now  discuss the  existence of isoperimetric regions in   $M$. We follow Morgan \cite{M1}. Let $\Omega_{\alpha}$ be a minimizing sequence for a fixed volume $V$. First we take a partition of $\mathbb{H}^3$ into congruent polyhedron $Q_j$ and consider only  those such that $Q_j \cap M\neq \emptyset$. Moreover, we choose $Q_j$ large enough so that $\Omega_\alpha$ does not contain any $Q_j$. Through  hyperbolic reflections  across the boundary of $M$, we can extend $\partial \Omega_{\alpha}$ such that its boundary  does not intersect $Q_j$. The number of reflections is independent of $\Omega_{\alpha}$.  The key observation is that $Q_j$ satisfies a relative isoperimetric inequality for some constant $\gamma$. As detailed in Morgan \cite{M1}, this isoperimetric inequality implies the existence of a number $\delta=\delta(\gamma,I(V))$ ($I$ is the isoperimetric profile of $M$) such that $vol(\Omega_{\alpha}\cap Q_i)> \delta V$ for some  of the $Q_i$. Choose  a rigid motion in $\mathbb{H}^3$ (not necessarily preserving $\partial M$) that brings the center of  $Q_i$ to a fixed point $O\in \mathbb{H}^3$. Standard compactness and regularity   applied to the sequence $\partial\Omega_{\alpha}$ passing through $O$ imply it will converge to a constant mean curvature surface $\partial\Omega$ enclosing volume $V_0\geq \delta V$. By the monotonicity formula,  $\partial \Omega$ must be bounded in $\mathbb{H}^3$ since it has finite area. It is enough to repeat the  process finitely many times to recover the  volume $V$; this also folows from the monotonicity formula  since the value of the constant mean curvature at each repetition does not change. Therefore, the minimizing sequence $\Omega_{\alpha}$ can be replaced by another which does not  drift off to infinity.
\vspace{0.2cm}

\noindent The situation is different when the hyperbolic distance between the planes is zero. For the region bounded by two vertical planes in the half-space model of $\mathbb{H}^3$ for example, there are no isoperimetric regions. If such set existed,  it would be either a half  geodesic ball or a tube about a geodesic. The latter is ruled out since there is no free boundary geodesic connecting the two planes. Moreover,  half geodesic spheres of a given  radius can always be constructed so that it is centered at one of the  planes and tangent to the other. Such configuration is not a critical point of the area functional under  volume constraints. In particular, there are no isoperimetric regions  $\mathbb{H}^3/G_a$ ($G_a$ composed of horizontal translations).
This   argument   also shows that    half geodesic balls  cannot be isoperimetric for all volumes in the slab bounded by two  geodesic planes of positive distance apart. Therefore, there  exists a critical volume $V_0$ for which tubes about geodesics  are the   isoperimetric surfaces when the enclosing volume satisfies $V\geq V_0$.

\end{document}